\documentclass[11pt]{amsart}
\usepackage{latexsym,enumerate}
\usepackage{amsmath,amsthm,latexsym,amssymb,mathrsfs,xcolor}
\usepackage{graphicx}
\usepackage{cite}
\numberwithin{equation}{section}

\newtheorem{theorem}{Theorem}

\newtheorem{lemma}{Lemma}
\newtheorem{corollary}{Corollary}

\newtheorem{remark}{Remark}

\newtheorem{definition}{Definition}

\numberwithin{theorem}{section}
\numberwithin{corollary}{section}
\numberwithin{lemma}{section}
\numberwithin{definition}{section}
\numberwithin{proposition}{section}
\numberwithin{remark}{section}

\newcommand{\Io}{0}
\newcommand{\I}{n}

\newcommand{\R}{\mathbb R}
\newcommand{\N}{\mathbb N}

\newcommand{\medint}{-\kern  -,375cm\int}
\newcommand{\dint}{\displaystyle\int}

\date{}

\title[]{Shape of extremal functions\\for weighted Sobolev-type inequalities}

\begin{document}
\author[F. Brock]{F. Brock$^2$}
\author[F. Chiacchio]{F. Chiacchio$^1$}
\author[G. Croce]{G. Croce $^3$}
\author[A. Mercaldo]{A. Mercaldo$^1$}

\normalsize \small
\renewcommand{\baselinestretch}{1.1}
\normalsize

\setcounter{footnote}{1}
\footnotetext{Universit\`a di Napoli Federico II, Dipartimento di Matematica e Applicazioni ``R. Caccioppoli'',
Complesso Monte S. Angelo, via Cintia, 80126 Napoli, Italy;\\
e-mail: {\tt francesco.chiacchio@unina.it,  mercaldo@unina.it}}

\setcounter{footnote}{2}
\footnotetext{
Martin-Luther-University of Halle, Landesstudienkolleg, 06114 Halle, Paracelsusstr. 22, Germany,  
e-mail: {\tt friedemann.brock@studienkolleg.uni-halle.de}
}
\setcounter{footnote}{2}
\footnotetext{
SAMM, UR 4543 Universit\'e Paris 1 Panth\'eon-Sorbonne, FR 2036 CNRS, 90, rue de Tolbiac, 75013 Paris, France,
e-mail: {\tt gisella.croce@univ-paris1.fr}
}

\begin{abstract} 
We study the shape of solutions to certain variational problems  in Sobolev spaces with weights 
that are powers of $|x|$. In particular, we detect situations when the extremal functions lack symmetry properties such as radial symmetry and antisymmetry. We also prove an isoperimetric inequality for the first nonzero eigenvalue of a weighted Neumann problem.

\medskip

\noindent
{\sl Key words: Rayleigh quotient, foliated Schwarz symmetry,  Bessel functions,  eigenfunction, breaking of symmetry}  
\rm 
\\[0.1cm]
{\sl 2000 Mathematics Subject Classification: 49J40, 49K20, 49K30}  
\rm 
\end{abstract}
\maketitle

\section{Introduction}
In this paper, we study the shape of minimizers of certain variational problems defined with power-type weights, and we exhibit some symmetry-breaking phenomena. These questions have attracted considerable interest in the literature.

Consider for example the Rayleigh quotient
\begin{equation}
\label{Rayleigh}
Q_{p,q, \gamma } (v):=  
\frac{\displaystyle{\int_B  |\nabla v|^p \, dx} }{\displaystyle
{\left( \int_{ B }  |v|^q  |x|^{\gamma }  \, dx \right) ^{p/q} }} , \quad v\in W^{1,p} (B) \setminus \{ 0\} ,
 \end{equation}
where $B$ denotes the unit ball, centered at the origin  in $\mathbb{R}^N$, 
$N\geq 2$, $p>1$ and $q\geq p$. 
\\
Variational problems of the type 
\begin{equation}
\label{smets} 
\inf \left\{ Q_{2,q, \gamma }(v): \ v\in W^{1,2}_0 (B) \setminus \{ 0\} \right\}\,, 
\end{equation}
with $2<q<2^* $,  have been studied in \cite{N2001}, \cite{KN}, 
 \cite{ByeonWang} and \cite{smets}. 
Similar problems with non-local operators have been studied in \cite{Shcheglova} (see also \cite{mercuri} and \cite{ABCMP}, where different weights are considered).
Typically the authors exihibit radial symmetry-breaking phenomena, showing that the minimizers are non-radial for some parameter values and radial for others. For example, as shown in \cite{smets}
for $2<q<2^*$, there exists a $\gamma^*>0$ such that for $\gamma>\gamma^*$ the minimizers are not radial. Instead, for  every $n\in \N$, there exists $\delta_n\to 0^+$, such that the unique minimizer
is  {radial}, if $\gamma\leq n$, $2\leq q\leq 2+\delta_n$.
\\ 
Furthermore, significant interest has been devoted to the shape of sign-changing minimizers of integral functionals, see for example \cite{Girao-Weth}, \cite{Weth}, \cite{BDNT},\cite{Parini-Weth} and \cite{BCGM}. 
In \cite{Girao-Weth}, Girao and Weth studied the symmetry properties of the minimizers of the problem
\begin{equation}
\label{girao-weth}
\inf \left\{ Q_{2,q,0} (v) \equiv \frac{\displaystyle{ \Vert |\nabla v| \Vert _2 ^2 }}{\displaystyle{ \Vert v\Vert _q ^2 }}: \, v
\in W^{1,2}(B)\setminus \{ 0\} ,\ \ \int_{B} v \, dx =0  \right\} 
\end{equation}
for $2\leq q<2^*$. 
They proved that the minimizers are foliated Schwarz symmetric.
This means that
they are  symmetric with respect to  reflection about some line $\R e$ and decreasing w.r.t. the angle $\arccos[\frac{x}{|x|}\cdot e]\in (0,\pi)$. 
Further, another interesting phenomenon related to the shape of the minimizers was pointed out for problem (\ref{girao-weth}):
if $p$ is close to 2, then any minimizer is antisymmetric w.r.t.  reflection
about the hyperplane $\{x\cdot e=0\}$.
In contrast to this, the minimizers are not anymore antisymmetric if $N=2$ and if $p$ is
sufficiently large.
A similar break of symmetry was already observed  in 
\cite{DGS}, \cite{BKN}, \cite{BK}, \cite{K}, \cite{N}, \cite{GGR} for the minimizers of
a one-dimensional problem,
$$
\inf \left\{ \frac{\displaystyle{ \Vert v' \Vert _p}}{\displaystyle{ \Vert v \Vert _q }} \, ,\quad v \in W^{1,p}((0,1))\setminus \{ 0\} , \  \ v(0)=v(1), \ \ \int_0^1 v\, dx=0 \right\} \,.
$$
More precisely, it has been shown that any minimizer is an antisymmetric function, if and only if $q\leq 3p$ (see also \cite{CD} and \cite{GN} for a more general constraint). 
For a recent survey article on the one-dimensional problem, we refer the reader to \cite{NS}.

In this paper, we study variational problems for Rayleigh-type quotients where both the numerator and the denominator carry weights which are powers of $|x|$.

Let $\Omega $ be a bounded domain in $\mathbb{R} ^N $, $N\geq 2 $, with Lipschitz boundary  containing the origin, and define 
\begin{equation}
\label{rayleighgeneral}
R_{p,q,\alpha , \gamma } (v) 
:=
\frac{\displaystyle{\int_{\Omega }   |\nabla v|^p \, |x|^{\alpha } \, dx }}{\displaystyle{\left(\int_{\Omega }   |v|^q \,  |x|^{\gamma }
\, dx   \right) ^{p/q}}} , \quad \ v\in W ^{1,p} (\Omega , |x|^{\alpha } , |x|^{\alpha} ) ,
\end{equation}
where
\begin{equation}
\label{basicconditions}
p,q \in [1, +\infty ) 
\end{equation} 
and the numbers $\alpha ,\gamma \in \mathbb{R}$ satisfy certain conditions. (The definitions of weighted function spaces, such as $W^{1,p} (\Omega , |x|^{\alpha }, |x|^{\beta } ) $, will be given in Section 2).    
We focus on two variational problems, one with Dirichlet boundary conditions and one with a mean value condition:
\begin{eqnarray}
\label{problemPD}
({\bf P}^D ) & & \inf \left\{ 
R_{p,q, \alpha , \gamma } (v): \ v\in W_0 ^{1,p} (\Omega , |x|^{\alpha } , |x|^{\alpha} ) \right\} =:\lambda ^D , 
\\
\label{problemPM}
({\bf P}^M ) & & \inf \left\{ 
R_{
2,q, \alpha , \alpha } (v) : \ v\in W ^{1,2} (\Omega , |x|^{\alpha } , |x|^{\alpha} ) , \ \ \int_{\Omega } |x|^{\alpha } v\, dx =0 \right\} =:\lambda ^M . 
\end{eqnarray}
We study the shape of solutions to these problems, and in particular, we detect situations when the extremal functions lack symmetry properties such as radial symmetry and antisymmetry. To state our results we will use the assumptions on $p, q, N, \alpha, \gamma$ that guarantee the existence of the minimizers given by Theorem \ref{2.4} in Section 2. More precisely, for problem 
$ ({\bf P}^{D}  ) $ we will prove the following result:
\begin{theorem}
\label{3.1}
Assume that $\Omega $ is a ball $B$, centered at the origin, $q\in (p, q_0 )$, where $q_0 $ is defined by (\ref{q0def}) below and  $0\leq \alpha <N(p-1) $. Then there exists a number $\gamma ^* \geq \alpha $ such that the minimizer of $ ({\bf P}^{D}  ) $ is not radially symmetric if $\gamma >\gamma ^* $.
\end{theorem}
This generalizes the result of \cite{smets} stated above, for which $p=2$ and $\alpha=0$.
We will prove that $\lambda^D<\inf \left\{ R_{\gamma } (v):\, v\in W_0 ^{1,p} (B, |x|^{\alpha }, |x|^{\alpha } ) \setminus \{ 0\}, \, v \, \mbox{ radial}  \right\} =\lambda_\gamma^{rad}.
$ This strict inequality will be achieved constructing 
a precise test function to prove that
$\lambda^D\leq C_0 \gamma^{-N+p+Np/q}$ and using that 
 for all $\gamma \geq \alpha$ there holds
 $$\lambda _{\gamma } ^{rad}\geq \left( \frac{\gamma + N}{\alpha +N} \right)^{p-1 + p/q} \cdot \lambda _{\alpha} ^{rad} \qquad \forall \gamma \geq \alpha .$$

For problem $({\bf P}^M )$, we will first prove that the minimizers are foliated Schwarz symmetric, that is, 
they are  axially symmetric with respect to an axis passing through the
origin and nonincreasing in the polar angle from this axis (as the second eigenfunction of the Laplacian with Dirichlet boundary conditions in the unit disk of $\R^2$):
\begin{theorem}
\label{4.1}
Assume 
$-N<\alpha <N$ and $q\in [2, q_0 )$, where $q_0 $ is given by (\ref{q0def}) below.
Then every minimizer of $({\bf P}^M )$ is foliated Schwarz symmetric w.r.t. some point $P\in {\mathbb S}^{N-1}$.
\end{theorem}
This result has been already obtained for the case $\alpha =0$ in \cite{Girao-Weth} and we adapt their technique to prove it. We will use the two point rearrangement as a tool (see Definition \ref{2ptsrearrangement}) as well as some standard regularity results of solutions to elliptic PDE, after studying the Euler equation. This strategy could be probably adapted to other variational problems.

Our other results 
about problem $({\bf P} ^M)$ will be established in dimension $N=2$. 
First we deal with the case $q=2$. We denote $J_{\nu_1}$  the Bessel function of order $\nu_1$ and we prove  the following result
\begin{theorem}
\label{5.1}
Let $\Omega =D$, where $D$ denotes the unit ball in $\mathbb{R}^2 $, $q =2$ and 
and $|\alpha|<2$. Then, if $u$ is 
a minimizer of $({\bf P} ^M)$, there holds  
\begin{equation}
\label{expressionofu}
u(x)=\varphi_1(r) (A_1\cos \theta+B_1\sin \theta)\,, \quad x= (x_1 , x_2 )= (r\cos \theta , r\sin \theta )\in B \, ,
\end{equation}
where 
$$
\varphi_1(r)=r^{-\alpha/2}J_{\nu_1}(\sqrt{2\lambda } r)\, , \, \quad 0\le r\le 1\, .
$$
Here $\varphi_1 $  is a solution to the problem
\begin{eqnarray}
\label{fi}
&&
\left\{ 
\begin{array}{l}
r^2\varphi_1''(r)+(\alpha+1)r\varphi_1'(r)+(2\lambda r^2 - 1)\varphi_1(r)=0\,,
\\
\varphi _1 (r)>0, \quad 0<r\le 1\,\,, 
\\
\varphi '_1(1)=0 \,,
\end{array}
\right.
\\
&&
\nonumber
\lambda =\frac{x_{\nu_1,1}^2}2\, , \quad \nu_1=\sqrt{1+\frac{\alpha^2}4} \, ,
\end{eqnarray}
$x_{\nu_1,1}$ is  the first  positive root of the equation
\begin{equation*} 
-\frac{\alpha}{2} 
J_{\nu_1}(x)
+
x J'_{\nu_1}(x)=0\,,
\end{equation*}
and $A_1 , B_1 \in \R $ are arbitrary constants.
\end{theorem} 

\noindent We explicitely remark that 
formula (\ref{expressionofu}) can be rewritten as
\begin{equation*}
\label{expressionofu1} 
u(x) = C\cdot \varphi _1 (r) \cdot \cos (\theta -\theta _0) , \quad x \in B, 
\end{equation*}
for some numbers $C\in \mathbb{R} $ and $\theta _0 \in [0, \pi ]$, that is, 
$u$ is foliated Schwarz symmetric.

\noindent The proof of this theorem relies on a very fine analysis of the solutions of the Euler equation satisfied by a minimizer, written in separable polar coordinates.
\\
In dimension 2, up to a rotation, the foliated Schwarz symmetry implies that every minimizer is symmetric (even) with respect to $x_1$ for every $q\geq 2$. Moreover the above theorem tells us that that  for $q=2$
every minimizer is  antisymmetric (odd) with respect to $x_2$. We will prove that this does not occur 
when $\alpha <0$ and $q$ is sufficiently large, in the following result:
\begin{theorem}
\label{thmbreaking}
Let $-2<\alpha<0 $ and $N=2$. Then there is a number $\widetilde{q} >2$, such that, if $q>\widetilde{q}$ and $u$ is a corresponding minimizer of $({\bf P} ^M)$ which is symmetric (even) w.r.t. to $x_1 $, then $u$ is not antisymmetric w.r.t. $x_2 $. 
\end{theorem}

\noindent The key point  in the proof of this result is to prove that 
$$\inf \left\{ R_{\alpha,q}(v): \  v\in W_0^{1,2} (B, |x|^\alpha, |x|^{\alpha }) \setminus \{ 0\}, v(x_1,-x_2)=-v(x_1,x_2) \right\}\,$$
tends to 0, as $q\to \infty$, in the same spirit as a result of \cite{Ren-Wei}. Unfortunately, the technique works only in dimension 2. However it is robust enough to have been used in \cite{Girao-Weth} and \cite{BCGM}.

Finally,  we study a shape optimization problem,
for $q=2$ and $\alpha >0$. More precisely, 
if $\Omega^{\sharp }$ denote the ball centered at the origin
such that 
\begin{equation*}
\int_{\Omega }\left\vert x\right\vert ^{\alpha }dx=\int_{\Omega ^{\sharp
}}\left\vert x\right\vert ^{\alpha }dx,
\end{equation*} 
and if we denote 
\begin{equation}
\mu _{1, \alpha} (\Omega)
:=\inf \left\{ \frac{
\dint_{\Omega }\left\vert \nabla v\right\vert ^{2}\left\vert x\right\vert
^{\alpha }dx}
{\dint_{\Omega }v^{2}\left\vert x\right\vert ^{\alpha }dx}:\
v\in W^{1,2}(B,|x|^{\alpha },|x|^{\alpha })\backslash \left\{ 0\right\} ,\ \
\int_{\Omega }v\left\vert x\right\vert ^{\alpha }dx=0\right\},  \label{mu_1}
\end{equation}
then the following result holds:
\begin{theorem} 
\label{TSW} 
Let $\Omega $ be a bounded Lipschitz domain in 
$\mathbb{R}^{N},$ with $N\geq 2,$ containing the origin, symmetric with respect to the origin and
let $\alpha \in (0,N).$ Then the infimum in \eqref{mu_1} is achieved and
\begin{equation} 
 \label{SW}
\mu _{1,\alpha}(\Omega)\leq \mu _{1,\alpha}(\Omega
^{\sharp }), 
\end{equation}
where equality holds if and only if $\Omega =\Omega ^{\sharp }$.
\end{theorem}

\noindent Clearly  $\mu _{1, 0} (\Omega)$ coincides with $\mu _{1}(\Omega )$,
the first nonzero eigenvalue  of the classical Neumann Laplacian.

\noindent Now let us briefly describe how Theorem \ref{TSW}  fits into the literature. Kornhauser and Stakgold conjectured in \cite{KS} that, among all planar simply
connected domains with fixed Lebesgue measure, the first nonzero eigenvalue
of the classical Neumann Laplacian achieves its maximum value if and only if the
domain is a disk.
This conjecture was later proved by Szeg\H{o} in \cite{S}. 
His proof uses the so-called ``harmonic transplantation", which, in turn, relies on the Riemann Mapping Theorem. Therefore, his result and techniques are confined to simply connected planar domains. In  \cite{W},  Weinberger generalized this result to any bounded smooth domain 
$\Omega $ in $\mathbb{R}^{N}$. It is worth noticing that
Weinberger's method has proven to be quite flexible. In fact, it has been
used, for example, in \cite{B}, \cite{AB}, \cite{CdB} 
and  \cite{ABD}. In fact, 
 some of its ideas, appropriately
adapted, are also used in the present paper. In Section 7, we will present the proof of Theorem \ref{TSW}, 
highlighting similarities and differences with Weinberger's proof. In particular, we will justify our additional assumption 
on the symmetry of $\Omega$. We do not see an obvious way to remove it
 and, therefore,  we leave this issue as a challenging open problem.

Let us outline the content of the next sections. Theorems \ref{3.1}, \ref{4.1}, \ref{5.1}, \ref{thmbreaking}, \ref{TSW} will be proved  in Sections 3, 4, 5, 6, 7 respectively.
Section $2$ contains several preliminary results. More precisely, we obtain continuous and compact embedding properties for some weighted function spaces, which are stated in Theorem \ref{functional_analysis}, Corollary \ref{functional_analysis_1} and Lemma \ref{poincare} and are based on an embedding theorem by T. Horiuchi (see \cite{Horiuchi}).  These properties allow us to prove  Theorem \ref{2.4} which gives the exi\-sten\-ce of solutions to the variational problems $({\bf P} ^D )$ and $({\bf P} ^M)$. Then we recall the definitions of the two-point rearrangement, foliated Schwarz symmetrization and foliated Schwarz symmetry given in Definitions \ref{2ptsrearrangement}, \ref{FSsymmetrization}, \ref{FSsymmetry} respectively. Moreover  we give some relations between these notions in Lemma \ref{lemmacavalieri} and Theorem \ref{lemma_FSS}.  
\\



\section{Preliminaries}
\label{prel}

\subsection{Embedding results and existence of solutions to $({\bf P} ^D)$ and $({\bf P} ^M )$}

In this subsection we assume that 
$\Omega$ is a bounded domain with Lipschitz boundary in $\R ^N $, $N\geq 1$,  containing the origin, $\alpha , \beta \in \R $ and $1\leq p\leq N$.
We denote by $L^p (\Omega , |x|^{\alpha })$ the weighted Lebesgue space of all measurable functions $u: \Omega \to \R $ with
$$
\Vert u\Vert _{p,\Omega , \alpha } := \left( \int_{\Omega }|u|^p   |x|^{\alpha } \, dx \right) ^{1/p} <\infty .
$$ 
The weighted Sobolev space $W^{1,p} (\Omega , |x|^{\alpha }, |x|^{\beta }  )$ is defined as the set of all functions $u\in L^p (\Omega , |x|^{\alpha }) $ having distributional derivatives $(\partial u / \partial x_i ) $, $i=1, \ldots ,N$, for which the norm
$$
\Vert u\Vert _{1,p, \Omega , \alpha , \beta } := \left[ \Vert u\Vert _{p, \Omega ,  \alpha  } ^p + \Vert |\nabla u| \Vert_{p, \Omega , \beta } ^p \right] ^{1/p} $$
is finite.  
It is well-known that, if 
\begin{equation}
 \alpha >-N,  \quad \beta >-N,
\label{alphabetacond1}
\end{equation} 
then $W^{1,p} (\Omega , |x|^{\alpha }, |x|^{\beta } ) $ is a reflexive Banach space, and if
\begin{equation}
\alpha < N(p-1), \quad  \beta <N(p-1) , 
\label{alphabetacond2}
\end{equation}
then
$C_0 ^{\infty} (\Omega ) \subset W^{1,p} (\Omega , |x|^{\alpha} , |x|^{\beta } )$, (see e.g. \cite{KufnerOpic}, p. 240 ff., and \cite{Leonardi}, p. 1054).
\\
Under the conditions (\ref{alphabetacond2}) and (\ref{alphabetacond1}) 
the space $W_0 ^{1,p} (\Omega , |x|^{\alpha } , |x|^{\beta})   
$ is defined as the closure of $C_0 ^{\infty } (\Omega )$ with respect to the norm $\Vert \cdot \Vert_{1,p, \Omega , \alpha , \beta } $. 
\\
\hspace*{0.5cm}
One can find a variety of embedding theorems for weighted Sobolev spaces into weighted Lebesgue spaces or into spaces of continuous functions in the literature (see, e.g., \cite{Il'in}, \cite{Maz'ya}, \cite{KufnerOpic}, \cite{Horiuchi}, \cite{Drabek-Kufner-Nicolosi}, \cite{Heinonen-etal}). 

The proof of the following result, which will be crucial in proving the existence of solutions to our variational problems, will be based on an embedding theorem by T. Horiuchi (see \cite{Horiuchi}).

 We use the notation $\hookrightarrow $ for continuous embedding and $\hookrightarrow \hookrightarrow$ for compact embedding. 
\begin{theorem}\label{functional_analysis}
Let $1\leq p < +\infty $ and $-N<\alpha \leq \gamma$, and define 
\begin{equation}
\label{q0def}
q_0 := \left\{ 
\begin{array}{ll}
+\infty & \mbox{ if }\ p\geq N \ \mbox{ and } \ \alpha \leq p-N 
\\[0.2cm]
\displaystyle{\frac{(2N-p+ \alpha )p}{N-p +\alpha }} &
\mbox{ if } \  p\geq N \ \mbox{ and }\  \alpha >p-N 
\\[0.3cm]
\displaystyle{\frac{Np}{N-p} } & \mbox{ if } N>p \ \mbox{ and } \ 
 \alpha \leq 0   
\\[0.3cm]
\displaystyle{\frac{(N+\alpha )p}{N-p +\alpha }} & \mbox{ if }\ N>p\ \mbox{ and }\  0< \alpha 
\end{array}
\right.
\quad .
\end{equation}
Then 
\begin{eqnarray}
\label{contemb}
&&
W^{1,p} (\Omega , |x|^{\alpha } , |x|^{\alpha } )
\hookrightarrow 
L^q (\Omega , |x|^{\gamma }) 
\\
\nonumber
&&
 \quad \mbox{ for every  \ $q\in [p, q_0]$, if $q_0 <+\infty $, and }
 \\
 \nonumber
 &&
 \quad \mbox{ for every  \ $q\in [p, +\infty )$, if $q_0 =+\infty $}
\end{eqnarray}
Furthermore,  
\begin{equation}
\label{compemb}
W^{1,p} (\Omega , |x|^{\alpha } , |x|^{\alpha } )
\hookrightarrow 
\hookrightarrow 
L^q (\Omega , |x|^{\gamma }) \quad \mbox{ for every } \ q\in [p, q_0).
\end{equation} 
\end{theorem}

 Lemma A below is a special case of  Theorem 3 of \S $3$ in  \cite{Horiuchi}, namely Cases A and B, with $k=1$, $j=0$ and $F=\{ 0\} $, (which implies $s=N$ - see the definition of the property $SP(s)$ on p. 373
of \cite{Horiuchi}).


\vspace{0.1 cm}
 
\noindent {\bf Lemma A:}
{\sl Let $\Omega $ be a domain in $\R ^N$ with $0\in \Omega $, $1\leq p\leq q$, $a>-N $ and $b>-N $. 
\\
{\bf (i) } Suppose that the following conditions hold:
\begin{eqnarray}
\label{abq1}
&& \frac{b}{q} \leq \frac{a}{p} ,
\\
\label{abq2} 
&& \frac{N-p+a}{p} \leq \frac{b+N}{q} \quad \mbox{ and }
\\
\label{abq3} 
&& \frac{b}{q} < \frac{N-p +a}{p}
\end{eqnarray}
Then 
\begin{equation}
\label{emb1}
W^{1,p} (\Omega , |x|^a, |x|^a) \hookrightarrow L^q (\Omega , |x|^b)  . 
\end{equation}
{\bf (ii) } Suppose that (\ref{abq1}) and 
\begin{equation}
\label{abq4}
  0\leq \frac{b}{q} = \frac{N-p+a}{p}
\end{equation}
hold. Then the embeddings (\ref{emb1}) hold for $p\leq q <+\infty $. 
}
\begin{corollary}
\label{2.1} 
The conditions (\ref{abq1})-(\ref{abq3}), respectively (\ref{abq1}) and (\ref{abq4}), are satisfied in the special cases below, which in turn leads to a corresponding range for the embedding (\ref{emb1}):
\\
{\bf 1. } If $p\geq N$, $a=p-N$ and $b\leq 0$, then (\ref{emb1}) holds for 
$p\leq q<+\infty $. 
\\
{\bf 2. } If $p\geq N$, $a>p-N$ and 
$-p+a\leq b\leq N-p+ a$, then (\ref{emb1}) holds for
$p\leq q \leq (N+b)p/ (N-p+a)$.
\\
{\bf 3. } If $p\geq N$ and  $b< N-p+a<0$, then  (\ref{emb1}) holds for 
$p\leq q < bp/ (N-p+a)$.
\\  
{\bf 4. } If $p<N$ and 
$b\leq a\leq p-N$, 
then  (\ref{emb1}) holds for 
$p\leq q\leq bp/a$.
\\
{\bf 5. } If $p<N$, $N-p+a >0$ and $b\leq a<0$, then (\ref{emb1}) holds for 
\\
$p\leq q \leq \min \{ bp/a;\  (b+N)p/(N-p+a)\} $.  
\\
{\bf 6. } If $p<N$, $a\geq 0$ and $-p+a \leq b\leq a$, then (\ref{emb1}) holds for 
$p\leq q \leq (b+N)p/(N-p+a) $. 
\end{corollary}
We are going to prove Theorem \ref{functional_analysis}.
\begin{proof} 
Let $-N <\alpha \leq a $ and $-N <b\leq \gamma $. Since $\Omega $ is bounded, there are positive constants $C_1 $ and $C_2 $, such that 
$$
|x|^a \leq C_1 |x|^{\alpha } , \ |x|^{\gamma } \leq C_2 |x|^b \quad \forall x\in \Omega \setminus \{0\} .
$$
This implies 
\begin{eqnarray}
\label{emb2} 
&& W^{1,p} (\Omega , |x|^{\alpha } , |x|^{\alpha} ) \hookrightarrow W^{1,p} (\Omega , |x|^a ,|x|^a) \quad \mbox{and}  
\\
\label{emb3}
&& L^q (\Omega , |x|^b ) \hookrightarrow L^q (\Omega , |x|^{\gamma } ).
\end{eqnarray}
Now we make the following choices of the numbers $a$ and $b$:
\\ 
{\bf (i) } If $p\geq N$  and $\alpha \leq p-N$, then we choose 
$a:= p-N$ and $b\in (-N, \min \{ 0; \gamma \} )$. 
In view of Corollary \ref{2.1}, {\bf 1.} we obtain (\ref{emb1}) for $p\leq q <+\infty $.
\\
{\bf (ii) } If $p\geq N$ and $\gamma \geq \alpha > p-N$,  then we choose 
$ a:=\alpha $ and $b:=\alpha $.
In view of Corollary \ref{2.1}, {\bf 2.} we obtain (\ref{emb1}) for $p\leq q\leq (2N-p+ a )p/(N-p+a )$.
\\  
{\bf (iii) } If $p<N$, $\alpha \leq 0$, then we choose 
$a:=\alpha $ and $b:= a N/(N-p)$. 
In view of Corollary \ref{2.1}, {\bf 5.} we obtain (\ref{emb1}) for $p\leq q\leq Np/(N-p) $.
\\
{\bf (iv) } If $p<N$ and $\alpha >0$, then we choose 
$a:= \alpha $ and $b:=  \alpha $.   
In view of Corollary \ref{2.1}, {\bf 6.} we obtain (\ref{emb1}) for $p\leq q \leq (N+a)p/(N-p+a ) $. 
\\
Using (\ref{emb2}) and (\ref{emb3}), this proves
(\ref{contemb}) with $q_0 $ given by (\ref{q0def}). The compact embeddings (\ref{compemb}) are trivial for $p=q$ and they follow from standard interpolation for $q<q_0 $.   
\end{proof}
\noindent
{\bf Remark 2.1.} 
{\sl Note that there is no continuous embedding of $W^{1,p} (\Omega , |x|^{\alpha } , |x|^{\alpha } )$ into $L^q (\Omega , |x|^{\alpha }) $ when $N>p$, $0\leq \alpha $ and 
$q>p(N+\alpha )/(N-p +\alpha )$.
To see this, choose $B_R \subset  \Omega $ and $u\in C_0 ^{\infty} (\Omega )$ with $ \mbox{supp}\, u\subset B_R $. Setting
$$
u_t (x) := u(t^{-1} \cdot x), \quad (0<t\leq 1),
$$
we have
\begin{eqnarray*}
&&
\Vert u_t \Vert _{p, \Omega , \alpha } ^p =t^{N+\alpha} \cdot  
\Vert u_t \Vert _{p, \Omega , \alpha } ^p, \ \ 
\Vert u_t \Vert _{q, \Omega , \alpha } ^q =t^{N+\alpha} \cdot  
\Vert u_t \Vert _{q, \Omega , \alpha } ^q\ \mbox{ and}
\\
&&
\Vert |\nabla u_t| \Vert_{p, \Omega, \alpha } ^p = t^{N+\alpha -p } \cdot    
\Vert |\nabla u| \Vert_{p, \Omega, \alpha } ^p.
\end{eqnarray*}
It follows that
$$
\frac{
\Vert u_t \Vert_{1, p, \Omega , \alpha  , \alpha } }{\Vert u_t \Vert_{q, \Omega , \alpha }} \to 0 \ \mbox{ as }\ t\to 0,
$$
and in particular,
$$
\frac{\Vert |\nabla u_t |\Vert _{p, \Omega , \alpha }}{\Vert u_t\Vert_{q, \Omega , \alpha }} 
 \rightarrow 0 \ \mbox{ as }\ t\to 0.
$$ 
From this the claim follows. }
\\[0.1cm]
By our assumptions,  $W_0 ^{1,p} (\Omega , |x|^{\alpha }, |x|^{\alpha })$ is a closed subspace of $W^{1,p} (\Omega , |x|^{\alpha }, |x|^{\alpha } )$. 
Hence we have the following
\begin{corollary}
\label{functional_analysis_1}
If $\alpha <N(p-1)$ and $\gamma <N(p-1)$, then the assertion of Theorem 2.1 holds with $W_0 ^{1,p} (\Omega , |x|^{\alpha } , |x|^{\alpha })$ in place of   
$W ^{1,p} (\Omega , |x|^{\alpha } , |x|^{\alpha })$.
\end{corollary}
We will also need the following Poincar\'{e}-type inequalities. 
\begin{lemma}\label{poincare}
Let $1\leq p\leq N$ and $-N<\alpha <N(p-1)$. Then there are positive constants $C_1 $, $C_2$, such that
\begin{eqnarray}
\label{poincare1} 
\Vert |\nabla u| \Vert_{p,\Omega , \alpha } 
 & \geq & C_1 \Vert u \Vert_{p,\Omega , \alpha } \quad \forall u\in W_0 ^{1,p} (\Omega , |x|^{\alpha }, |x|^{\alpha }), \quad \mbox{and}
 \\
 \label{poincare2} 
\Vert |\nabla u| \Vert_{p,\Omega , \alpha } 
 & \geq & C_2 \Vert u- u_{\Omega } \Vert_{p, \Omega , \alpha } \quad \forall u\in W ^{1,p} (\Omega , |x|^{\alpha }, |x|^{\alpha }),
 \\[0.1cm]
 \nonumber & & \mbox{where }\ u_{\Omega } := \frac{\displaystyle\int_{\Omega } u|x|^{\alpha } \, dx }{\displaystyle\int_{\Omega } |x|^{\alpha } \, dx }\, .
 \end{eqnarray}
 \end{lemma}
\begin{proof}
Since $-N<\alpha <N(p-1)$, the weight $|x|^{\alpha }$  belongs to the Muckenhoupt class $A_p$. Hence it is also $p$-admissible, which means that (\ref{poincare1}) holds,  (see \cite{Heinonen-etal}, Chapter 15 and formula (1.5)). 
\\ 
The proof of (\ref{poincare2}) can be carried out analogously as in the unweighted case $\alpha =0$, using the compactness of the embedding of $W^{1,p} (\Omega , |x|^{\alpha }, |x|^{\alpha }) $ into $L^p (\Omega , |x|^{\alpha})$, (compare \cite{Evans}, § 5.8.1, proof of Theorem 1).  
\end{proof}
We conclude this subsection 
with the following existence result.
\begin{theorem}
\label{2.4}
Let $1\leq p\leq N$, $-N<\alpha <N(p-1)$, $\gamma \geq \alpha $ and $q\in [p, p_0 )$. Then the problems $({\bf P}^D ) $ and $({\bf P} ^M ) $  have solutions and the corresponding minima $\lambda ^D $ and $\lambda ^M $ are positive.
\end{theorem}
\begin{proof}
 From Theorem \ref{functional_analysis} and Lemma \ref{poincare} we deduce that there are positive constants $C'$ and $C^{\prime \prime}$ such that
\begin{eqnarray*}
&& \Vert |\nabla u | \Vert_{p, \Omega , \alpha } \geq C' \Vert u\Vert _{q, \Omega , \gamma } \,,
\qquad \forall
u\in W_0 ^{1,p} (\Omega , |x|^{\alpha } , |x|^{\alpha } ) , 
\\
&&
\Vert |\nabla u | \Vert_{2, \Omega , \alpha } \geq C^{\prime \prime} \Vert u\Vert _{q, \Omega , \alpha } \,,
\qquad \forall 
u\in W ^{1,2} (\Omega , |x|^{\alpha } , |x|^{\alpha } ) \ \mbox{ with }\ u_{\Omega } =0 ,
\end{eqnarray*}
and the assertions follow by standard arguments. 
\end{proof}

\subsection{Foliated Schwarz symmetry}
In this subsection we assume that $\Omega $ is a domain that is radially symmetric w.r.t.  the origin.
In other words, $\Omega $ is either an annulus, a ball, or the exterior of a ball in $\mathbb{R} ^N $.
If $u: \Omega \to \mathbb{R} $  is a measurable function, 
we will for convenience always extend $u$ onto $\mathbb{R} ^N $ by setting $u(x)=0 $ for $x\in \mathbb{R} ^N \setminus \Omega  $.
\begin{definition}
\label{2ptsrearrangement}
Let 
${\mathcal H} _0 $ be the family of open half-spaces $H$ in $\mathbb{R} ^N $ such that $0\in \partial H $. 
For any  $H \in {\mathcal H} _0 $, let $\sigma _H $ denote the reflection in $\partial H $. 
We write
$$
\sigma _H u (x) := u(\sigma _H x), \quad x\in \mathbb{R} ^N .
$$
The two-point rearrangement w.r.t. $H$  is given by
\[
u_H (x) := 
\begin{cases}
\max \{ u(x); u(\sigma _H x ) \} & \mbox{ if }\ x \in H ,
\\
\min \{ u(x); u(\sigma _H x ) \} & \mbox{ if } \ x\not\in H .
\end{cases}
\]  
\end{definition}
Note that one has $u= u_H $ if and only if  $u(x)\geq u(\sigma _H x )$ for all $x\in H$.
Similarly, $\sigma _H u= u_H $  if and only if   $u(x)\leq u(\sigma _H x) $ for all $x\in H$.
\\[0.1cm]
\hspace*{0.3cm}We will make use of the following properties of the two-point rearrangement.

\begin{lemma}\label{lemmacavalieri}
Let $H\in {\mathcal H} _0 $.
\begin{enumerate}
\item
If $A\in C([0, +\infty ),\mathbb{R})$, $u:\Omega \to \mathbb{R} $ is measurable and $A (|x|,u)\in L^1(\Omega)$, then  $A (|x|, u_H)\in L^1(\Omega)$ and 
$\displaystyle
\int_\Omega A (|x|, u)\, dx =\int_\Omega A (|x|, u_H)\, dx \, .
$
\item
If $u\in W^{1,2} (B, |x|^{\alpha } )$, then 
$\displaystyle 
\int_B  |\nabla u|^2 \, |x|^\alpha \, dx
=
\int_B  |\nabla u_H|^{2} \, |x|^\alpha \, dx\,.
$
\end{enumerate}
\end{lemma}
\begin{proof}
We observe that $|\sigma _H x|= |x|$, we have for a.e. $x\in H\cap \Omega $. Therefore
$$ 
A(|x|, u(x))+ A(|\sigma _H x| ,u(\sigma _H x))= 
A(|x|, u_H (x))+ A(|\sigma _H x| ,u_H (\sigma _H x))$$
and
$$
 |x|^\alpha |\nabla u(x)|^2 + |\sigma _H x|^\alpha|\nabla u(\sigma _H x)|^2 = 
|x|^\alpha|\nabla u_H (x)|^2  + |\sigma _H x|^\alpha|\nabla u_H (\sigma _H x)|^2.
$$
It is now sufficient to integrate these two equalities on $\Omega\cap H$.
\end{proof}
\noindent
\hspace*{0.5cm}Now we recall the definition of foliated Schwarz symmetrization of a function. Such a function is   axially symmetric with respect to an axis passing through the
origin and nonincreasing in the polar angle from this axis.
\begin{definition}
\label{FSsymmetrization}
If $u:\Omega \to \mathbb{R} $ is measurable, the {\sl foliated Schwarz symmetrization } $u^* $ of $u$ is defined as  the (unique) function satisfying the following properties:
\begin{enumerate}
\item
there is a function $w : [0, +\infty ) \times [0, \pi ) \to \mathbb{R } $, 
$w= w (r, \theta )$, which is nonincreasing in 
$\theta $, and 
\[
u^* (x) = w\left( |x| , \arccos (x_1 /|x| ) \right) , \quad (x \in \Omega );
\]  
\item
$
{\mathcal L} ^{N-1}  \{ x:\, a < u(x) \leq b,\, |x| =r \} = {\mathcal L} ^{N-1} \{ x:\, a < u^* (x) \leq b ,\, |x|=r \}$  
for all $a,b \in \mathbb{R}$ with $a<b$, and $r\geq 0$.
\end{enumerate}
\end{definition} 
\begin{definition}\label{FSsymmetry}
Let $P_N$ denote the point $(1,0, \ldots , 0)$, the 'north pole' of
the unit sphere ${\mathbb S} ^{N-1} $. 
We say that $u$ is {\sl foliated Schwarz symmetric  w.r.t.  $P_N$} if 
$u=u^*$ - that is, $u$ depends solely on $r$ and on $\theta $ - the 'geographical width' -, 
and is nonincreasing in $\theta $. 
\\
We also say that $u$ is {\sl foliated Schwarz symmetric w.r.t. a point $P\in {\mathbb S}^{N-1} $ } 
if there is a rotation about the origin $\rho $ such that $\rho (P_N)= P$, and 
$u(\rho (\cdot)) = u^* (\cdot )$.
In other words, a function $u:\Omega\to \R$ is foliated Schwarz symmetric with respect to $P$ if, for every $r>0$ and $c\in \R$,
the restricted superlevel set $\{x: |x|=r, u(x)\geq c\}$ is equal to $\{x: |x|=r\}$ or a geodesic ball in the
sphere $\{x: |x|=r\}$ centered at $rP$. In particular, $u$ is axially symmetric with respect to the axis $\R P$.
\\
\noindent 
Moreover
a measurable function  $u:\Omega \to \mathbb{R}$ is foliated Schwarz symmetric w.r.t. 
$P\in {\mathbb S}^{N-1} $ iff $u=u_H $ for all $H\in {\mathcal H} _0 $ with $P\in H$. 
\end{definition}
The next result was proved in \cite{BCGM}. It will be used in Section \ref{FSS}.
\begin{theorem}
\label{lemma_FSS} 
 Let $u\in L^{p} (\Omega )$ for some $p\in [1, +\infty )$, 
and assume that for every $H\in {\mathcal H} _0 $ one has either $u= u_H $, or $\sigma _H u= u_H $.
Then $u$ is foliated Schwarz symmetric w.r.t. some point $P\in {\mathbb S}^{N-1} $.
\end{theorem}

\section{Non-radiality for solutions to problem $({\bf P}^D)$}
\noindent
\hspace*{0.5cm}In this section we study  problem $({\bf P} ^D)$ when $\Omega =: B$   is the unit ball centered at the origin.
\\
Let  $\alpha $, $p$ and $q$ be fixed.  
For any number $\gamma \geq \alpha $ we write for convenience  
\begin{eqnarray*}
&& R_{ \gamma } (v):= R_{p,q, \alpha , \gamma} (v) 
, \quad v\in W_0 ^{1,p} (B, |x|^{\alpha }, |x|^{\alpha } ), 
\\
&& ({\bf P} _{\gamma }) := ({\bf P} ^D) \ \mbox{ and } \ \lambda _{\gamma } := \lambda ^D. 
\end{eqnarray*}
Denote
\begin{equation}
\label{lambdaradial}
\lambda _{\gamma } ^{rad} := \inf \left\{ R_{\gamma } (v):\, v\in W_0 ^{1,p} (B, |x|^{\alpha }, |x|^{\alpha } ) \setminus \{ 0\}, \, v \, \mbox{ radial}  \right\} .
\end{equation}
As already mentioned in the Introduction, we are going to prove Theorem \ref{3.1}.

We merely need to show that
\begin{equation}
\label{lambdaineq1}
\lambda _{\gamma } <\lambda _{\gamma } ^{rad},
\end{equation}
if $\gamma $ is large enough. 
Our approach is similar as in \cite{Girao-Weth}.
For the proof of (\ref{lambdaineq1}) we need two
lemmata.
\begin{lemma}
\label{3.3}
There exists a number $C_0 >0$, independent of $\gamma $, such that for all $\gamma \geq 3$,
\begin{equation}
\label{lambdaineq2}
\lambda _{\gamma } \leq C_0 \cdot \gamma ^{-N + p+ Np/q} .
\end{equation}
\end{lemma}
\begin{proof}
Let $U\in W_0 ^{1,p} (B)$ be a positive first eigenfunction for the Dirichlet $p$-Laplacian in $B$, with eigenvalue $\underline{\lambda }$, that is
\begin{equation}
\label{eigenplaplace1}
\left\{ 
\begin{array}{l}
  -\Delta _p U \equiv -\nabla (|\nabla U|^{p-2} \nabla U) = \underline{\lambda } U^{p-1} \ \mbox{ in } \ B,
  \\
\\
 U=0 \ \mbox{ on } \ \partial B.
\end{array}
\right.
\end{equation}
We extend $U$ by zero outside $B$ and set 
$x^{\gamma } := (1-\gamma ^{-1} , 0, \ldots ,0)$ and $U_{\gamma } (x):= U(\gamma (x-x^{\gamma }))$. Then $U_{\gamma } \in W_0 ^{1,p} (B_{1/\gamma } (x^{\gamma }))$ and
\begin{equation}
\label{Ugamma1}
\int _{B} |\nabla U_{\gamma } |^p \, dx = \underline{\lambda } \gamma ^p \int
_{B} (U_{\gamma})^p \, dx .
\end{equation}   
It follows that 
\begin{eqnarray}
\label{ineq1}
\int_B  |\nabla U_{\gamma } |^p \, |x|^{\alpha } \, dx 
 & \leq & 
\int_B |\nabla U_{\gamma } |^p \, dx 
=
 \underline{\lambda }  \gamma ^p \int_B ( U_{\gamma } )^p \, dx.
\end{eqnarray}
On the other hand, we have by the minimality property of $\lambda _{\gamma} $ and in view of H\H{o}lder's inequality,
\begin{eqnarray}
\label{ineq2}
 \int_B |\nabla U_{\gamma }| ^p \, |x|^{\alpha}  \, dx 
 & \geq &  
\lambda _{\gamma } \left( \int _B  (U_{\gamma } )^q \, |x|^{\gamma } \, dx \right) ^{p/q} 
\\
\nonumber
 & = &  
\lambda _{\gamma } \left( \int _{B_{1/\gamma } (x^{\gamma }) }  
 (U_{\gamma } )^q \, |x|^{\gamma } \, dx \right) ^{p/q}
\\
\nonumber 
 & \geq &
\lambda _{\gamma } ( 1-2 \gamma ^{-1} ) ^{\gamma p/q} \cdot 
\left( 
\int _{B_{1/\gamma } (x^{\gamma })}   
(U_{\gamma } )^q \, dx 
\right) ^{p/q} 
\\
\nonumber 
 & \geq &
\lambda _{\gamma }  ( 1-2 \gamma ^{-1} ) ^{\gamma p/q} \cdot 
\left( 
\int _{B_{1/\gamma } (x^{\gamma } ) } dx \right) ^{(p/q)-1 } \cdot \int_B  (U_{\gamma } )^p \, dx  
\\
 \nonumber
 & = & \lambda _{\gamma }  ( 1-2 \gamma ^{-1} ) ^{\gamma p/q} \cdot \gamma ^{N-Np/q} \cdot (\omega _N ) ^{(p/q)-1} \cdot \int_B  (U_{\gamma } )^p \, dx,
\end{eqnarray} 
where $\omega _N $ denotes the Lebesgue measure of $B$. 
Now (\ref{ineq1}) and (\ref{ineq2}) yield
\begin{equation}\label{ineq3}
\lambda _{\gamma } \leq \underline{\lambda } \gamma ^{-N+p + Np/q}  (1-2\gamma ^{-1} ) ^{-\gamma p/q} (\omega _N) ^{1-p/q} \leq C_0 \gamma ^{-N +p + Np/q}, 
\end{equation}
where $C_0 $ does not depend on $\gamma$.
\end{proof}
\begin{lemma}\label{3.4}
There holds for all $\gamma \geq \alpha$,
\begin{equation}
\label{ineq4}
\lambda _{\gamma } ^{rad}\geq \left( \frac{\gamma + N}{\alpha +N} \right)^{p-1 + p/q} \cdot \lambda _{\alpha} ^{rad} .
\end{equation} 
\end{lemma}
\begin{proof}
Let $u\in W_0 ^{1,p} (B, |x|^{\alpha }, |x|^{\alpha })$ be a radial function, such that
\begin{equation}
\label{lambdaR1}
\lambda_{\gamma } ^{rad}= R_{\gamma } (u). 
\end{equation}
We write $u=u(r)$, where $r=|x|$. Setting $z:= r^{(\gamma +N)/(\alpha + N)}$ and $w(z):= u(r)$, and taking into account that $\alpha+N-p>0$ and $\gamma \geq \alpha $, we calculate
\begin{eqnarray}
\label{enum}
\int_B  |\nabla u|^p \,   |x|^{\alpha }  \,  dx 
 & = & N\omega _N \int_0 ^1 r^{\alpha +N-1} |u'(r)|^p \, dr 
\\
\nonumber
 & = & N\omega _N \left( \frac{\gamma +N}{\alpha + N} \right) ^{p-1} \int_0 ^1 z^{\alpha + N-1 - (\gamma - \alpha )(N+\alpha -p)/ (\gamma +N)}
  |w' (z)|^p \, dz
\\
\nonumber 
 & \geq  & N\omega _N  \left( \frac{\gamma +N}{\alpha + N} \right) ^{p-1} \int_0 ^1 z^{\alpha +N-1 } |w' (z)|^p \, dz ,
\end{eqnarray}
and
\begin{eqnarray}
\label{denom}
\int_B  |u|^q \,   |x|^{\gamma }    \, dx 
 & = &  N\omega _N \int_0 ^1 r^{\gamma +N-1} |u|^q \, dr 
  =  N\omega _N \cdot \frac{\alpha + N}{\gamma+N}  \int_0 ^1 z^{\alpha + N-1} |w|^q \, dz \, .
\end{eqnarray}
From (\ref{lambdaR1}), (\ref{enum}) and (\ref{denom}) we obtain
\begin{eqnarray*}
\lambda_{\gamma } ^{rad} 
 & \geq &  
\left( \frac{\gamma +N}{\alpha + N} \right) ^{p-1+ p/q} 
\frac{N\omega _N \int_0 ^1 z^{\alpha +N-1} |w'(z)|^p\, dz }{\left( N\omega _N \int_0 ^1 z^{\alpha + N-1} |w|^q \, dz \right) ^{p/q} } 
\\
\nonumber 
 & \geq & 
\left( \frac{\gamma +N}{\alpha +N} \right) ^{p-1+ p/q} \cdot
\lambda _{\alpha } ^{rad} . 
\end{eqnarray*}

\end{proof}
\noindent
{\sl Proof of Theorem \ref{3.1}. } 
One has from the inequalities
(\ref{lambdaineq2}) and (\ref{ineq4})
$$
\frac{\lambda_{\gamma}^{rad}}{\lambda_{\gamma}}\geq 
\frac{\lambda_{\alpha }^{rad} \cdot \left( \frac{\gamma +N}{\alpha + N} \right) ^{p-1+ p/q}}{C_0 \cdot \gamma ^{-N + p+ Np/q}} .
$$
Since $q>p$, we have that
$$
p-1 + \frac{p}{q} > -N + \frac{Np}{q} +p.
$$
It follows that
$$
\frac{\lambda_{\gamma}^{rad}}{\lambda_{\gamma}}\longrightarrow + \infty \ \mbox{ as } \ \gamma \to +\infty ,
$$ 
and (\ref{lambdaineq1}) follows if $\gamma $ is large enough.
$\hfill \Box $

\section{Foliated Schwarz symmetry of solutions to problem $({\bf P} ^M)$ }\label{FSS}
We are going to prove Theorem \ref{4.1}, about the minimizers of $({\bf P}^M )$, thus generalizing the case
$\alpha =0$ studied in \cite{Girao-Weth}. 

\smallskip
\noindent
{\sl Proof of Theorem \ref{4.1}. }  
We divide the proof into steps. We denote $B$  any ball centered in the origin, and for convenience we write $ \lambda ^M = \lambda $.
\\
\noindent {\sl Step 1:}
Let $H\in {\mathcal H}_0$, and let $u$ be a minimizer of $({\bf P}^M)$.
Then  by assuming that 
$
\Vert u\Vert _{q,B , \alpha } =1
$, then $u$ satisfies  the following Neumann boundary value problem for the Euler equation given by
\begin{equation}\label{euler1}
\left \{
\begin{array}{ll}
-\nabla \left( |x|^{\alpha}\nabla u\right) =   2\lambda   |x|^{\alpha} |u|^{q-2} u +\mu |x|^{\alpha}&\quad \hbox{in } B
\\
 &\\
 \dfrac{\partial u}{\partial \nu }=0 & \quad \hbox{on } \partial B
 \end{array}
 \right.
 ,
\end{equation}
for some $\mu \in \R$, where $\nu $ denotes the exterior unit normal.\\ 
By the assumption on $q$ and classical  regularity theory,  
we deduce that 
$u$ is bounded in $B\setminus B_{\epsilon} $ for every $\epsilon >0$, and then 
$u\in C^2 (\overline{B} \setminus \{ 0\} )$.
\\
\noindent 
On the other hand, the following equalities hold by Lemma 2.6:
 $$
u_H\ne 0, \quad u_H\in W ^{1,2} (B, |x|^\alpha, |x|^{\alpha }) ,\quad 
\int_\Omega  u_H  \, |x|^{\alpha} \, dx=0 , \quad \|u_H\|_{q,B,\alpha}=1, 
$$
$$
\int_B |u|^q \,  |x|^\alpha   \, dx
=
\int_B |u_H|^q    \,   |x|^\alpha   \,  dx\,,\quad
\int_B  u \,  |x|^\alpha  \,  dx =
\int_B  u_H  \,  |x|^\alpha  \,  dx \,, 
$$
$$
\int_B  |\nabla u|^2 \, |x|^\alpha \, dx
=
\int_B  |\nabla u_H|^2   \,  |x|^\alpha  \,  dx\,.
$$
and  therefore we get
$$
R_{2,q,\alpha, \alpha }(u)=R_{2,q,\alpha, \alpha}(u_H)\,.
$$
Hence, $u_H $ is a minimizer, too, so that it satisfies 
the same Euler equation satisfied by $u$ and boundary Neumann condition , i.e.
\begin{equation}\label{euler2}
\left \{
\begin{array}{ll}
 -\nabla \left( |x|^{\alpha}\nabla u_H\right) =  2 \lambda |x|^{\alpha} |u_H|^{q-2} u_H +\mu |x|^{\alpha}&\quad \hbox{in } B
\\ &\\
\dfrac{\partial u_H}{\partial \nu}=0 & \quad \hbox{on } \partial B\,.
 \end{array}
 \right.
\end{equation}

Moreover $u_H$ satisfies the same regularity properties of $u$, that is $u_H\in C^2 (\overline{B} \setminus \{ 0\} )$.
\\
\noindent 
{\sl Step 2:}
Define 
$v:= u-u_H $ and note that $v\geq 0 $ in $B\cap H$.
Then $v\in C^2 (\overline{B} \setminus \{ 0\} )$ satisfies the following linear elliptic equation
\begin{equation}\label{difference}
- \nabla (|x|^{\alpha } \nabla v ) =2  \lambda |x|^{\alpha } m(x) v\,, \quad \hbox{in } B\setminus \{ 0\}  
\end{equation}
where 
$$
m(x) := \left\{
\begin{array}{ll}
\dfrac{ 
 |u|^{q-2} u- |u_H|^{q-2} u_H  }{v(x)} 
 & \quad \mbox{if $v(x)\ne 0$},
 \\&\\
0 & \quad \mbox{if $v(x)=0$}
\end{array}
\right.   
$$ 
Since $u,u_H\in C^2 (\overline{B} \setminus \{ 0\} )$, $m(x)$ is a bounded function in $B\setminus B_{\epsilon} $ for every $\epsilon >0$.
\\
We claim that for every half-space $H$ with $0\in \partial H$ there holds one of the following:
\begin{enumerate}
\item
$\sigma _H u \equiv u_H$ on $H$,
\item 
$u\equiv u_H$  on $H$.
\end{enumerate}
If (1) holds, we are done. 
Note that (1) implies that $u(x) \leq \sigma_H u(x)$ on $H$.
Hence, if (1) does not hold, then there is a point $x_0 \in H$ with
$u(x_0 ) > \sigma _H u (x_0 )$. 
Since $u$ is continuous, there is a neighborhood $W$ of $x_0 $ with $W\subset H$, such that
$u(x) > \sigma _H u(x )$ on $W$, which also implies 
$u(x)\equiv u_H (x) $ in $W$, that is, $v\equiv 0$  in $W$. 
We may apply
the Principle of Unique Continuation to (\ref{difference})
to conclude that $v\equiv 0$, that is,
$u\equiv u_H $ throughout $H$. 
In other words, (2) holds.
This proves the claim. 
\\
Finally  by Theorem \ref{lemma_FSS} 
 this implies that $u$ is - up to a rotation about the origin - foliated Schwarz symmetric with respect to some point $P\in {\mathbb S}^{N-1} $.
$\hfill \Box $
\begin{remark}
The above result holds in the case of an annulus centered at the origin, too.
\end{remark}

\begin{remark}\label{trasf}
We explicitly observe that, by using the transform $\rho=r^\beta$, with $\beta=1+\alpha(N-p)$, we could reduce problems $({\bf P} ^D)$ and $({\bf P} ^M)$ to the classical ones with $\alpha=0$ and  Theorems \ref{3.1}, \ref{4.1} above could be obtained as in the classical case. Anyway we have given a different  proof of the two results  adapted to the weighted problems.
\end{remark}

\section{Shape of solutions to problem $({\bf P} ^M )$ for $q=2$ and $N=2$}\label{caseq=2}
\noindent
In this section we show the explicit expression of the  solutions to problem $({\bf P}^M)$ in the case $q=2$ and $N=2$. This will be useful to prove symmetry properties of the minimizers.
\\
First we recall some properties of Bessel functions  (see, for example, \cite{AS}).
\subsection*{5.1. A few properties of Bessel functions}
\noindent
\\[0.1cm]
It is well-known that   Bessel functions $J_\nu$, $Y_\nu$ of order $\nu$ of the first and second kind, are linearly independent for any value of $\nu$ (see, for example, \cite{AS} p. 358). 
The following relation holds true 
\begin{equation}\label{jy}
Y_\nu(r)=\frac{J_\nu(r)\cos(\nu \pi)-J_{-\nu}(r)}{\sin(\nu \pi)},
\end{equation}
for non integer $\alpha$
and where the right-hand side is replaced by its limiting value whenever $\nu$ is an integer. 
Moreover, $J_\nu$ satisfies the following fundamental recurrence relation
\begin{equation}\label{consrecbess}
rJ_\nu'(r)-\nu J_{\nu}(r)=-rJ_{\nu+1}(r),  \qquad r \in \R.
\end{equation}
If we denote by $j_{\nu,h}, j_{\nu,h}'$ the zeros of $J_\nu, J_\nu'$, respectively, then
$$
\nu\le j'_{\nu,1}<j_{\nu,1}<j'_{\nu,2}<....
$$
and
\begin{equation}\label{zeros}
 j_{\nu,1}<j_{\nu+1,1}<j_{\nu,2}<....
\end{equation}
In  \cite{AS} (Prop. 9.1.9, p. 360), the following identities can be found
\begin{equation}
\label{jnu}
\begin{aligned}
J_\nu(r)&=\frac{\left(\frac 1 2 r\right)^\nu}{\Gamma(\nu+1)} \prod_{h=1}^\infty \left(1-\frac{r^2}{j_{\nu,h}^2}\right), \quad \nu \ge 0
\\
J_\nu'(r)&=\frac{\left(\frac 1 2 r\right)^{\nu-1}}{2\Gamma(\nu)} \prod_{h=1}^\infty \left(1-\frac{r^2}{(j_{\nu,h}')^2}\right), \quad \nu > 0.
\end{aligned}
\end{equation}
Finally, we will deal with the equation 
\begin{equation} \label{eq1bis}
-\frac{\alpha}{2} 
J_{\nu}(x)
+
x J'_{\nu}(x)=0\,\,\,\,\,\,x\geq 0\,.
\end{equation}
The roots of this equation have been studied in \cite{Landau}.
We rewrite it as
\begin{equation} \label{eq2bisF}
F_{\nu}(x)=\frac{\alpha}{2} \,,
\end{equation}
where 
\begin{equation} \label{fnu}
F_{\nu}(x)
=
x\frac{J'_{\nu}(x)}{J_{\nu}(x)}
=
\nu-x\frac{J_{\nu+1}(x)}{J_{\nu}(x)}
=
-\nu+x\frac{J_{\nu-1}(x)}{J_{\nu}(x)}
\,,
\end{equation}
for any positive $x$ which is not a zero for $J_{\nu}$.
Here we used the property
\begin{equation}\label{proprietaBessel}
zJ_{\nu}'(z)=\nu J_{\nu}(z) - z J_{\nu+1}(z)\,.
\end{equation}
We emphasize that the positive zeros of $J_{\nu}(x)$ are not solutions of equation \eqref{eq2bisF}. 
\\
\noindent It is proved in \cite{Landau} that, for any $\nu>-1$, the function $F_{\nu}(x)$ decreases from the value $\nu$ at $x=0$ to $-\infty$ at $x=j_{\nu,1}$, the first positive zero of the function $J_{\nu}(x) $, jumping to $+\infty$ as $x$ moves past $j_{\nu,1}$ and decreases to $-\infty$ at $x=j_{\nu,2}$ and so on (see Figures 4 in \cite{Landau}).
\\
\noindent Let $x_{\nu, k}$, $k=1,2,...$, be the positive roots of the equation
\begin{equation*} \label{eq1bis}
-\frac{\alpha}{2} 
J_\nu(x)
+
x J'_\nu(x)=0\,,
\end{equation*}
ordered in increasing order.
In  \cite{Landau}  the behaviour of  $x_{\nu, k}$ is described as the order $\nu$ varies over the entire range of all real values. In particular, at page 196, it is proved that
$$
\frac d{d\nu} x_{\nu, k}>0
\,\, \, 
\hbox{whenever }
\,\, 
F'_{\nu}(x_{\nu, k})<0\,. 
$$
\subsection*{5.2. Explicit expression of the eigenfunctions in dimension 2, for $q=2$}
\noindent
\\[0.1cm]
For convenience we again write $\lambda := \lambda ^M $ for the infimum in problem $({\bf P} ^M)$. The main result of this section is the
\\
\noindent
{\sl Proof of Theorem \ref{5.1}.} 
Let  $u$ be a minimizer to problem $({\bf P} ^M)$. Then
$u$ solves  the  Neumann boundary value problem for the Euler equation given by
\eqref{euler1}. 
 It is easy to see  that in this case $\mu=0$. Indeed, one can use $u$ as test function in the Euler equation and integrate on $D$. The constraint on the weighted average of $u$  on the right-hand side gives the conclusion.
\\
By using polar coordinates and standard separation of variables, we can write $u$ as
$$u(x_1,x_2)=v(r,\theta)=\varphi(r)w(\theta)\,, \qquad 0\le r\le 1\,, \quad 0\le \theta\le 2\pi\,.
$$
where
 $\varphi (r)$ and $w(\theta)$ are solutions to the following problems respectively: 
\begin{eqnarray}
&&
\label{fi1}
\left\{
\begin{array}{l}
r^2\varphi''(r)+(\alpha+1)r\varphi'(r)+(2\lambda r^2+C)\varphi(r)=0\,,\ \ 0<r\le 1\,,
\\
 \varphi'(1)=0 
 \end{array}
 \right.
 ,
\\
\label{w1}
&&
\left\{
\begin{array}{l}
w''(\theta)-C w(\theta)=0\,,\ \ 0<\theta\le 2\pi\,,
\\
 w(0)=w(2\pi) 
\end{array}
\right.
.
\end{eqnarray}
That is 
$$
w_n(\theta)=\left\{
\begin{array}{ll}
A_0, & n=0
\\
A_n \cos(n\theta)+B_n\sin(n\theta), &n\ge1
\end{array} 
\right.
$$
for any constant $A_0$, $A_n$, $B_n \in \R$ and
\\
$$
\varphi_n(r)= r^{-\alpha/2}[c_1J_{\nu_\I}(\sqrt{2\lambda } r)+c_2Y_{\nu_\I}(\sqrt{2\lambda } r)],
$$
where $c_1$, $c_2$ are arbitrary constants and  $J_{\nu_\I}(r)$, $Y_{\nu_\I}(r)$, with $\nu_\I={\sqrt{n^2+\frac{\alpha^2}{4}}}$, are Bessel functions of first and second kind respectively. 
\\
Since the solution $u$ must belong to the weighted space $L^{2} (B, |x|^\alpha)$, necessarily $c_2=0$. Indeed by \eqref{jy} and \eqref{jnu}, for any fixed $\nu > 0 $ and 
$r\rightarrow 0^+$, it holds that
$$
J_{\nu}(r) \sim c_\nu  r^{\nu}\, \hbox{ and } \, 
Y_{\nu}(r) \sim c_\nu r^{-\nu}\,.
$$
Therefore the integral of $u_n=\varphi_n(r)w_n(\theta)$, that is,
$$
\int_B |u_n|^2 |x|^\alpha\, dx=\int_0^{2\pi}w_n^2(\theta)\,d\theta\int_0^1\varphi_n^2(r)r^{\alpha+1}\,dr\,,
$$
is finite if, and only if,
$$
-\alpha-2\nu_\I +\alpha+1>-1\,,
$$
that is,
$$
\nu_\I <1\, .
$$
But such a condition is not verified if $n\ge 1$. This justifies the choice of $c_2=0$ for $n\ge 1$.
\\
For $n=0$, condition $\nu_\I <1$ is equivalent to $|\alpha|<2$. Moreover , since 
\begin{equation}
\label{A_JY}
J_{\nu }(r)\sim r^{\nu },\text{\ \ }Y_{\nu }(r)\sim r^{-\nu }\text{, \ }
\frac{d}{dr}J_{\nu }(r)\sim r^{\nu -1},\text{ \ }\frac{d}{dr}Y_{\nu }(r)\sim
r^{-\nu -1},
\end{equation}
an analogous argument shows that
$$
\int_B |\nabla u_0|^2 |x|^\alpha\, dx=A_0\int_0^1|  \varphi_0 ^{\prime} (r)|^2r^{\alpha+1}\,dr\,
$$
is finite if, and only if
$
-\alpha-2-2\nu_0 +\alpha+1>-1\,.
$
But such a condition is not verified. This justifies the choice of $c_2=0$ also for $n=0$.
\\
We now impose the Neumann condition $\varphi'_n(1)=0$ in the expression
$$
\varphi_n(r)= c_1r^{-\alpha/2}J_{\nu_\I}(\sqrt{2\lambda } r)\,.
$$ 
\noindent 
An easy calculation gives, for any $0<r<1$:
$$
\varphi'_n(r)=
-c_1\frac{\alpha}{2} r^{-\frac\alpha2-1}
J_{\nu_\I}(\sqrt{2\lambda } r)
+
c_1 r^{-\frac\alpha2}\sqrt{2\lambda } J'_{\nu_\I}(\sqrt{2\lambda} r)\,.
$$
Therefore the Neumann condition $\varphi'_n(1)=0$ is equivalent to 
\begin{equation}\label{bdN}
-\frac{\alpha}{2} 
J_{\nu_\I}(\sqrt{2\lambda })
+
\sqrt{2\lambda } J'_{\nu_\I}(\sqrt{2\lambda })=0\,.
\end{equation}
This means that $\sqrt{2\lambda } $ is a positive root of the equation
\begin{equation} \label{eq1}
-\frac{\alpha}{2} 
J_{\nu_\I}(x)
+
x J'_{\nu_\I}(x)=0\,
\end{equation}
or equivalently
\begin{equation} \label{eq2bis}
F_{\nu_\I}(x)=\frac{\alpha}{2} \,,
\end{equation}
according to \eqref{eq2bisF}.
Let us consider now for any fixed $n\in \N\cup\{0\}$ the positive roots $x_{\nu_\I, k}$, $k=1,2...$ of the equation \eqref{eq1}. For $n=0$ the smaller positive root is
$$x_{\nu_\Io, 1}, \quad  \hbox{with } \quad \nu_\Io=\frac{|\alpha|}2\,.
$$
For the value $\nu_\Io=\frac{|\alpha|}2$ and definition \eqref{fnu} of function $F_{\nu_\I}(x)$,  equation \eqref{eq2bis} becomes
\begin{equation*} 
\frac{|\alpha|}2-x\frac{J_{\nu_\Io+1}(x)}{J_{\nu_\Io}(x)}=\frac{\alpha}{2} \,, \quad \hbox{if } \alpha>0\,,
\end{equation*}
or
\begin{equation*}  
-\frac{|\alpha|}2+x\frac{J_{\nu_\Io-1}(x)}{J_{\nu_\Io}(x)}=\frac{\alpha}{2} \,, \quad \hbox{if } \alpha<0\,.
\end{equation*}
This implies that the positive root $x_{\nu_\Io, 1}$ of equation \eqref{eq2bis} coincides with the zero $j_{\nu_{\Io+1},1}$ of  the Bessel function $J_{\nu_\Io+1}$, when $\alpha>0$ and coincides with the zero $j_{\nu_{\Io-1},1}$ of  the Bessel function $J_{\nu_\Io-1}$, when $\alpha<0$. 
\\
\noindent Assume $\alpha>0$. By previous described properties of $x_{\nu, k}$, we know that
$$
x_{\nu_1, 1}<j_{\nu_1,1}
$$
and by properties of zero's Bessel functions, since $\nu_1=\sqrt{1+\frac {\alpha^2}4}<\nu_\Io+1=\frac{|\alpha|}2+1$,  it results
$$
j_{\nu_1,1}<j_{{\nu_\Io+1},1}\equiv x_{\nu_\Io, 1}\,.
$$
\\
\noindent Assume $\alpha<0$. In such a way $x_{\nu_\Io, 1}=j_{\nu_{\Io-1},1}$ (with $\nu_{\Io}-1>-1$) is the smallest positive root of equation \eqref{eq2bis} and therefore $\sqrt{2\lambda } = j_{\nu_{\Io-1},1}$. But such a root  cannot be considered. Indeed in this case we choose  $n=0$. Moreover the minimizer $u(x_1,x_2)=A_0\varphi_0(r)= A_0r^{-\alpha/2}J_{\nu_\Io}(j_{\nu_{\Io-1},1} r)$ must have zero weighted mean value, while by the following equality 
(see \cite{GR}, p.707 n.6.556 (9)) 
we get
$$
\int_0^1r^{1-\nu_\Io}J_{\nu_\Io}(j_{\nu_{\Io-1},1} r)\,dr=\frac{(j_{\nu_{\Io-1},1})^{\nu_{\Io-2}}} 
{2^{\nu_{\Io-1}}\Gamma(\nu_\Io)}-(j_{\nu_{\Io-1},1})^{-1}
J_{\nu_\Io-1}(j_{\nu_{\Io-1},1} )=\frac{(j_{\nu_{\Io-1},1})^{\nu_{\Io-2}}} 
{2^{\nu_{\Io-1}}\Gamma(\nu_\Io)}\ne 0\,.
$$
\\
We conclude that  in both cases the smaller positive root of equation \eqref{eq2bis} is given by $x_{\nu_1, 1}$. This implies that 
\begin{equation}\label{xnu11}
\sqrt{2\lambda}=x_{\nu_1,1}\,, \quad \nu_1=\sqrt{1+\frac {\alpha^2}4}
\end{equation}
and $\varphi_1(r)$ is the corresponding solution to problem \eqref{fi}. 
\\
\noindent Finally the uniqueness (up to rotations and multiples)
of the function $u(x_1,x_2)=v(r, \theta)$ is a consequence of standard properties of completeness. $\hfill \Box $
\\
\section{Break of anti-symmetry of solutions to problem $({\bf P} ^M)$ for $N=2$ and large $q$}
\label{caseqlarge}
\noindent
\hspace*{0.5cm}In this section we give conditions in the two-dimensional case, such that the minimizers of problem $({\bf P} ^M)$ fail to be antisymmetric.  
\\
We recall that 
the foliated Schwarz symmetry proved in 
 Section 4 implies that, up to a rotation about the origin, a minimizer $u(x_1,x_2)$  is symmetric (even) in the variable $x_1$, for any $q\geq 2$. We are now going to analyse the behaviour of $u$  with respect to the other variable,
$x_2$. Note that, for $q=2$, 
formula
 (\ref{expressionofu}) 
implies that, if $u$ is even in the variable $x_1$, then $u$ is antisymmetric (odd) w.r.t. $x_2$.
\\
Readapting a technique of \cite{Girao-Weth}, we  prove in this section that, for $-2<\alpha <0$ and sufficiently large $q$, if $u$ is symmetric w.r.t. the variable $x_1 $, then $u$ is {\sl not} antisymmetric with respect to $x_2$. 
\\
\hspace*{0.5cm}
In the sequel let $B\subset \mathbb{R}^2$  denote the ball of radius 1 centered at the origin, and let $\lambda _{\alpha , q} := \lambda ^M$ be the corresponding infimum in problem $({\bf P} ^M)$.

The key point  in the proof of Theorem \ref{thmbreaking}  is a result by Ren and Wei (see Lemma 2.2 in \cite{Ren-Wei}), where it is shown that if one considers the  Rayleigh quotient $R_{0,q}$ in the space of $W^{1,2}_0(B)$ functions, 
the corresponding eigenvalue tends to $0$ as the parameter $q$ of the denominator goes to infinity. We  prove that the same  behaviour holds  for our eigenvalue $\lambda _{\alpha,q}$. 
\begin{lemma}\label{firstlimit}
Let  $\Omega $ be a bounded domain in $\R ^N $ containing $0$ and $-2<\alpha <0$. Further, let  
$$
\lambda^0_{\alpha,q}(\Omega):=\inf \left\{\frac{ \displaystyle \int_\Omega  |\nabla v|^2   |x|^{\alpha}\, dx }{ \displaystyle \left( \int_{ \Omega } |v|^q  |x|^{\alpha} \, dx \right) ^{2/q} } :\, v\in W_0^{1,2} (\Omega, |x|^\alpha, |x|^{\alpha} ) \setminus \{ 0\} \right\}, \quad q\geq 2 .
$$
Then 
$\lambda^0_{\alpha,q}(\Omega)\to 0$
as
$q\to \infty$.
\end{lemma}
\begin{proof}
Choose $R>0$ and $x_0 \in \Omega $ such that $B_{2R} (x_0 )\subset \Omega $ and $0\not\in B_{2R} (x_0 )$. For $q\geq 1$ we define $w_q :\R ^2 \to \R $ by 
$$
w_q(x)=\left\{
\begin{array}{ll}
q  & | \ \ 0\leq |x|\leq R e^{-q}
\\[0.1cm]
 \ln \frac{R}{|x|} 
   & | \ \ R e^{-q}\leq |x|\leq R
\\[0.1cm]
0 & |\ \ |x|\geq R\,.
\end{array}
\right.
$$
It has been shown in \cite{Ren-Wei}, Lemma 2.2, that 
\begin{equation}
\label{limvq}
\lim_{q\to + \infty} \frac{\displaystyle\int_{B_R (0)} |\nabla w_q |^2 \, dx }{\left( \displaystyle\int_{B_R (0) } 
|w_q |^q \, dx \right) ^{2/q} } =0.
\end{equation}    
Now let $u_q\in W _0 (\Omega , |x|^{\alpha }, |x|^{\alpha})$ be defined by
$$
u_q (x) := w_q (x-x_0), \quad x\in \Omega .
$$ 
In view of our assumptions there are positive constants $C_1 ,C_2 $ such that
$$
C_1 \leq |x|^{\alpha } \leq C_2\,, \quad \forall x\in B_R (x_0 ) .
$$
Together with (\ref{limvq}) we finally obtain
\begin{eqnarray*}
\lambda ^0 _{\alpha , q}
 &\leq &
\frac{
\displaystyle\int_{B_R (x_0 )}  |\nabla u_q |^2 \, |x|^{\alpha } \, dx }{
\left( 
\displaystyle\int_{B_R (x_0) } 
|u_q | ^q \,  |x|^{\alpha }   \,   dx \right) ^{2/q}
} 
\\
 & \leq & 
 C_2 
 \cdot 
 \left( C_1 \right) ^{-2/q} 
 \cdot 
 \frac{\displaystyle\int_{B_R (0)} 
 |\nabla w_q |^2 \, dx }{
 \left( 
 \displaystyle\int_{B_R (0) } 
|w_q |^q \, dx 
\right) ^{2/q} 
}  \longrightarrow 0 \ \mbox{ as }\ q\to +\infty .
\end{eqnarray*}   
\end{proof}
A direct consequence of the above lemma is the following  result
\begin{corollary}
\label{secondlimit}
Let 
$-2<\alpha<0$ and 
$$
\lambda^{as}_{\alpha,q}(B):=\inf \left\{ R_{\alpha,q}(v): \  v\in W_0^{1,2} (B, |x|^\alpha, |x|^{\alpha }) \setminus \{ 0\}, v(x_1,-x_2)=-v(x_1,x_2) \right\}\,.
$$
Then $\lambda^{as}_{\alpha,q}(B)\to 0$, as $q\to \infty$.
\end{corollary}
\begin{proof}
Let $u$ be a function realizing $\lambda^0_{\alpha,q}(B^+)$, where $B^+$ is the upper half part of the unit ball in the plane.
We  define
$$
w(x_1,x_2)=\left\{
\begin{array}{ll}
u ( x_1 , x_2 ) & | \ \  (x_1,x_2)\in B^+
\\
-u(x_1,-x_2)& | \ \  (x_1,x_2)\in B\setminus B^+
\end{array}
\right.
$$
and use it as a test function.
By Lemma \ref{firstlimit} this gives
$$
\lambda^{as}_{\alpha,q}(B)
\leq 
\frac{\displaystyle \int_B 
|\nabla w|^2   |x|^{\alpha}\, dx }{\displaystyle \left( \int_{ B } |w|^q  |x|^{\alpha} \, dx \right)
^{2/q} }=2^{1-\frac q2} \lambda^0_{\alpha,q}(B^+)\longrightarrow 0 \ \mbox{ as }\  q\to +\infty .
$$
\end{proof}
\medskip
Now we prove the main result of this section, Theorem \ref{thmbreaking}.
\\[0.1cm]
{\sl Proof of Theorem \ref{thmbreaking}.}
We define a particular test function $\widetilde{u}_q$ for $\lambda_{\alpha,q}(B)$
to prove that
$
\lambda_{\alpha,q}(B)<\lambda^{as}_{\alpha,q}(B)
$.
Let $v_q$ be a function such that $v_q(x_1,-x_2)=-v_q(x_1,x_2)$ realizing $\lambda^{as}_{\alpha,q}(B)$, such that $\displaystyle \int_B |\nabla v_q|^2  |x|^\alpha\, dx =1$.
We define
$$
\overline{u}_q(x_1,x_2)=\left\{
\begin{array}{ll}
v_q, & (x_1,x_2)\in B^+
\\
0&  (x_1,x_2)\in B\setminus B^+\,.
\end{array}
\right.
$$
We observe that
\begin{equation}
\label{normegradient}
\int_B |\nabla \overline{u}_q|^2 |x|^\alpha\, dx=\frac{1}{2}\, ,
\end{equation}
and
\begin{equation}
\label{normaLqoverlineu}
\int_B |\overline{u}_q|^q |x|^\alpha\, dx=\frac{1}{2} \int_B |{v}_q|^q  |x|^\alpha\, dx= \frac{1}{2} [\lambda^{as}_{\alpha,q}(B)]^{-q/2}\,.
\end{equation}
We now use  
$$
\widetilde{u}_q  :=  \overline{u}_q - d , \quad \mbox{where  }\ d := \frac{1}{\int_B |x|^{\alpha} \, dx } \int_B \overline{u}_q  |x|^\alpha\, dx,
$$
as a  test function for $\lambda_{\alpha,q}(B)$.
We have, by (\ref{normegradient}),
$$
\lambda_{\alpha,q}(B)
 \leq  
\frac{\displaystyle \int_B |\nabla  \overline{u}_q|^2 |x|^\alpha\, dx}
{\displaystyle \left[\int_B  \left| \overline{u}_q  - d 
 \right|^q \,  |x|^{\alpha}  \,   dx\right]^{2/q}} \, .
$$
By the triangle inequality and (\ref{normaLqoverlineu})
we get
\begin{eqnarray}
\label{lastlambda}
\lambda_{\alpha,q}(B)
&\leq &
\displaystyle \frac{1/2}{
\displaystyle 
\left[ 
\left( 
\int_B 
|\overline{u}_q  |^q \,  |x|^{\alpha }   \, dx    \right) ^{1/q} 
- d 
\left( 
\int_B |x|^{\alpha } \, dx 
\right) ^{1/q} 
\right] ^2   }
\\
\nonumber
&= &
\displaystyle \frac{1/2}{\displaystyle \left[\left[\frac{1}{2} \left[\lambda^{as}_{\alpha,q}(B)\right]^{-q/2}\right]^{ 1/q} - \displaystyle \left|\int_B \overline{u}_q  |x|^\alpha\, dx \right|\left[\int_B |x|^{\alpha}\, dx \right]^{(1/q)-1} \right]^{2}}\, ,
\\
\nonumber
&= &
\frac{(1/2)^{-(2/q)+1}}
{\left[1 - \displaystyle \left|\int_B \overline{u}_q  |x|^\alpha\, dx \right|\left[\displaystyle \int_B |x|^{\alpha} \, dx\right]^{(1/q)-1}\left[\lambda^{as}_{\alpha,q}(B)\right]^{1/2}\right]^{2}}\cdot \lambda^{as}_{\alpha,q}(B)\, .
\end{eqnarray}
By Lemma 2.3, (\ref{poincare2}), and Theorem 2.1  
there exists a positive constant $C$, independent of $q$,  such that
$$
\left|\int_B \overline{u}_q \ |x|^\alpha, dx\right|\leq C\,\,\,\,\,\,\,\,\forall\, q \geq 2 \,.
$$
Since $\lambda^{as}_{\alpha,q}(B)\to 0$ as $q\to \infty$ by Corollary \ref{secondlimit}, the denominator in the last line of (\ref{lastlambda}) tends to $1$, as $q\to \infty$. Therefore, for $q$ sufficiently large
$$
\lambda_{\alpha,q}(B)\leq \frac{2}{3} \cdot \lambda^{as}_{\alpha,q}(B)<\lambda^{as}_{\alpha,q}(B)\,.
$$
This
shows the breaking of anti-symmetry. $\hfill \Box $ 
\section{A weighted Szeg\H{o}-Weinberger inequality}
Throughout this section  we will denote
by $B_{R}$ the ball in $\mathbb{R}^{N}$, with $N\geq 2$, centered at the
origin with radius $R$ and we will assume that $\alpha \in (0,N)$.
Note that when $\alpha$ lies in this interval, 
Theorem \ref{functional_analysis} ensures that 
 $W^{1,2} (\Omega, |x|^\alpha , |x|^{\alpha })$ 
is compactly embedded in $L^{2}(\Omega
,\left\vert x\right\vert ^{\alpha }),$ 
  for any Lipschitz bounded domain 
$\Omega $ in $\mathbb{R}^{N}$, containing the origin.
Therefore $\mu_{1,\alpha}(\Omega)$, defined in \eqref{mu_1},
coincides with the first nonzero eigenvalue of the problem
\begin{equation}
	\left\{ 
	\begin{array}{ll}
		-\nabla \left( \left\vert x\right\vert ^{\alpha }\nabla u\right) 
		=\mu 
		\left\vert x\right\vert ^{\alpha }u & \text{in } \  \Omega  \\[0.1cm] 
		\dfrac{\partial u}{\partial \nu }=0 & \text{on } \  \partial \Omega ,
	\end{array}
	\right.   \label{P}
\end{equation}
where $\nu$ denotes the outer normal to $\partial \Omega$.
\\
We recall that, for any bounded domain $\Omega $ in $\mathbb{R}^{N}$, we denote   
 by  $\Omega ^{\sharp }$ the ball centered at the origin, whose  radius $r^{\sharp
}$ is such that
\begin{equation*}
\left\vert \Omega \right\vert _{\alpha }:=
\int_{\Omega }\left\vert x\right\vert ^{\alpha }dx=\int_{\Omega ^{\sharp
}}\left\vert x\right\vert ^{\alpha }dx=\frac{N\omega _{N}}{N+\alpha }\left(
r^{\sharp }\right) ^{N+\alpha }.
\end{equation*}
For the sequel we need a few notion and results from the theory of weighted rearrangements.

Let $u$ be a real measurable function defined in $\Omega$.
The decreasing rearrangement and the increasing rearrangement of $u(x)$, with
respect to the measure $\left\vert x\right\vert ^{\alpha }dx$, are defined
 by 
\begin{equation*}
u^{\ast }\left( s\right) =\inf \left\{ t\geq 0:\left\vert \left\{ x\in
\Omega :\left\vert u(x)\right\vert >t\right\} \right\vert _{\alpha }\leq
s\right\} \text{, \ }s\in \left( 0,\left\vert \Omega \right\vert _{\alpha }
\right] 
\end{equation*}
and
\begin{equation*}
u_{\ast }\left( s\right) =u^{\ast }\left( \left\vert \Omega \right\vert
_{\alpha }-s\right) \text{, \ }s\in \left[ 0,\left\vert \Omega \right\vert
_{\alpha }\right) 
\end{equation*}
respectively.

\noindent Moreover, the $N$-dimensional radially decreasing and radially increasing
rearrangements of $u(x)$, with respect to the measure $\left\vert
x\right\vert ^{\alpha }dx$, are defined as follows
\begin{equation*}
u^{\sharp }(x)=u^{\sharp }(\left\vert x\right\vert )=u^{\ast }\left( \frac{
N\omega _{N}}{N+\alpha }\left\vert x\right\vert ^{N+\alpha }\right) \text{ \
for }0<\left\vert x\right\vert \leq r^{\sharp }
\end{equation*}
and
\begin{equation*}
u_{\sharp }(x)
=
u_{\sharp }(\left\vert x\right\vert )
=
u_{\ast }\left( \frac{
N\omega _{N}}{N+\alpha }
\left\vert x\right\vert ^{N+\alpha }\right)
 \text{ \
for }
0 \leq \left\vert x\right\vert < r^{\sharp },
\end{equation*}
respectively.
\noindent
Clearly, for any nonnegative function $h : \Omega^{\sharp} \to \mathbb{R}$ that is radial and radially non-increasing (respectively, non-decreasing), it holds that
\[
h(x) = h^{\sharp}(x) \quad \text{(respectively, } h(x) = h_{\sharp}(x)\text{)}.
\]

The Hardy-Littlewood inequality   
states that for every pair of
measurable functions $u$ and $v$ defined on $\Omega $, it holds that

\begin{gather}
\label{HL}
\int_{\Omega ^{\sharp }}u^{\sharp }(x)v_{\sharp }(x)
\left\vert
x\right\vert ^{\alpha }dx
=
\int_{0}^{\left\vert \Omega \right\vert _{\alpha
}}u^{\ast }\left( s\right) v_{\ast }\left( s\right) ds
\leq
 \int_{\Omega
}\left\vert u\left( x\right) v\left( x\right) \right\vert 
\left\vert
x\right\vert ^{\alpha }dx\leq    \\
\text{   \ \ \ \ \ \ \ \ \ \ \ \ \ \ \ \ \ \ \ \ \ \ \ \ \ \ \ \ }
\int_{0}^{\left\vert \Omega \right\vert _{\alpha }}u^{\ast }\left(
s\right) v^{\ast }\left( s\right) ds
=
\int_{\Omega ^{\sharp }}u^{\sharp
}(x)v^{\sharp }(x)
\left\vert x\right\vert ^{\alpha }dx.  \notag
\end{gather}

Note that the inequalities (\ref{HL}) can be found for the uniform Lebesgue measure in \cite{Kes}, pp. 11--16. But it is well-known that they carry over to the weighted case, see e.g. \cite{BS}, p. 44.

As anticipated in the introduction, this section is devoted to the proof of the 
Szeg\H{o}-Weinberger-type inequality for $\mu_{1,\alpha}(\Omega)$ 
contained in Theorem \ref{TSW}.

\bigskip

The first step in proving the above-mentioned result is to show that 
$\mu_{1,\alpha}(B_{R})$ 
is an $N$-fold degenerate eigenvalue and a corresponding
set of eigenfunctions is in the form
\begin{equation*}
G(\left\vert x\right\vert )\frac{x_{i}}{\left\vert x\right\vert }\text{ for }
i=1,...,N,
\end{equation*}
for some suitable function $G$.
To this aim it is convenient to rewrite problem $(\ref{P})$, when $\Omega
=B_{R},$ in polar coordinates as follows
\begin{equation}
\left\{ 
\begin{array}{ccc}
-\dfrac{1}{r^{N-1}}\dfrac{\partial }{\partial r}\left( r^{N-1}\dfrac{
\partial u}{\partial r}\right) -\dfrac{1}{r^{2}}\Delta _{\mathbb{S}
^{N-1}}(\left. u\right\vert \mathbb{S}_{r}^{N-1})-\dfrac{\alpha }{r}\dfrac{
\partial u}{\partial r}=\mu_{1,\alpha}(B_{R}) \, u & \text{in} & B_{R} \\ 
&  &  \\ 
\dfrac{\partial u}{\partial r}=0 & \text{on} & \partial B_{R},
\end{array}
\right.   \label{P_R}
\end{equation}
where $\mathbb{S}_{r}^{N-1}=\partial B_{r}$, $\left. u\right\vert \mathbb{S}
_{r}^{N-1}$ is the restriction of $u$ on $\mathbb{S}_{r}^{N-1}$ and,
finally, $\Delta _{\mathbb{S}^{N-1}}(\left. u\right\vert \mathbb{S}
_{r}^{N-1})$ is the standard Laplace-Beltrami operator relative to the
manifold $\mathbb{S}_{r}^{N-1}.$
\\
It is well known that the solutions of the eigenvalue problem (\ref{P_R}) can be found via separation of variables. Writing 
$u(x)=Y(\theta )f(r)$ and plugging it into the equation in (\ref{P_R}), with $\theta
\in \mathbb{S}_{1}^{N-1}$, we get
\begin{equation*}
-\dfrac{Y}{r^{N-1}}\left( r^{N-1}f^{\prime }\right) ^{\prime }-\dfrac{f}{
r^{2}}\Delta _{\mathbb{S}^{N-1}}(Y)-\dfrac{\alpha }{r}Yf^{\prime }=\mu_{1,\alpha}(B_{R}) Yf
\end{equation*}
and in turn
\begin{equation*}
\dfrac{1}{fr^{N-3}}\left( r^{N-1}f^{\prime }\right) ^{\prime }
+
\alpha r\dfrac{f^{\prime }}{f}
+
\mu_{1,\alpha}(B_{R}) r^{2}
=-\dfrac{\Delta _{\mathbb{S}^{N-1}}(Y)}{Y
}=\overline{k}
\end{equation*}
Since the last equality is fulfilled if and only if 
\begin{equation*}
\overline{k}=k(k+N-2)\text{ \ with }k\in \mathbb{N}_{0}:=\mathbb{N}\cup
\left\{ 0\right\} 
\end{equation*}
(see e.g. \cite{Mu}), we have that
\begin{equation}\label{eqf}
f^{\prime \prime }+\frac{N-1+\alpha }{r}f^{\prime }+\mu_{1,\alpha}(B_{R}) f-\frac{k(k+N-2)}{
r^{2}}f=0\text{ \ with }k\in \mathbb{N}_{0}.
\end{equation}
 Hence  the eigenfunctions $\mu _i $ of problem (\ref{P_R}) are either purely radial
\begin{equation}
u_{i}(r)=f_{0}(\mu _{i};r),\text{ \ if }k=0,  \label{f0}
\end{equation}
or in the form
\begin{equation*}
u_{i}(r,\theta )
=
f_{k}(\mu _{i};r)Y(\theta ),\text{ \ if }k\in \mathbb{N}.
\end{equation*}
Denote $\mu:= \mu_{1,\alpha}(B_{R})$. Let us explicitely remark that equation \eqref{eqf} can be rewritten as
\begin{equation}\label{eqf1}
f^{\prime \prime }+\frac{\beta+1}{r}f^{\prime }+\left[\mu -\frac{k(k+N-2)}{
r^{2}}\right]f=0\,,
\end{equation}
with  $\beta=N-2+\alpha$ and $k\in \mathbb{N}_{0}$. 
As in Section 5 we deduce that solutions 
to equation \eqref{eqf} are given by
$$
f_k(r)=r^{-\frac\beta2}
\left( c_{1}J_{\nu_k}(\sqrt{\mu }r)+c_{2}Y_{\nu_k}(\sqrt{\mu }r)\right)
$$
where $c_1$, $c_2$ are arbitrary constants and 
$$
\nu_k=\sqrt{\frac{\beta^{2}}{4}+k(k+N-2)}=
\sqrt{\frac{(N-2+\alpha
)^{2}}{4}+k(k+N-2)}\,.
$$
Moreover the solutions $f_k$ belonging to 
 $W^{1,2}(B_R, |x|^\alpha,   |x|^\alpha)$ are obtained by choosing $c_2=0$, i.e.
 $$
f_k(r)= c_{1}r^{-\frac\beta2}
J_{\nu_k}(\sqrt{\mu }r)\,, \qquad 0<r<R\,.
$$
In the sequel we will denote by $\tau _{n}(R),$ with $n\in \mathbb{N}_{0},$
the sequence of eigenvalues of (\ref{P_R}) whose corresponding
eigenfunctions are purely radial, i.e. in the form (\ref{f0}). Clearly in
this case the first eigenfunction is constant and the corresponding
eigenvalue $\tau _{0}(R)=0.$ We will denote by $\upsilon _{n}(R),$ with $
n\in \mathbb{N},$ the remaining eigenvalues of (\ref{P_R}). We finally
arrange the eigenvalues in such a way that the sequences $\tau _{n}(R)$ and $
\upsilon _{n}(R)$ are increasing.
\\
Our weighted  Szeg\H{o}-Weinberger-type inequality relies on the following
\begin{lemma}\label{For_SW} 
The following inequality holds for every $R>0$:	
\begin{equation*}
\upsilon _{1}(R)<\tau _{1}(R).
\end{equation*}
\end{lemma}
\begin{proof}
We recall that $\tau _{1}:=\tau _{1}(R)$ is the first nonzero
eigenvalue of
\begin{equation}
\left\{ 
\begin{array}{ccc}
g^{\prime \prime }+\dfrac{N-1+\alpha }{r}g^{\prime }+\tau g=0 & \text{in} & 
(0,R) \\ 
&  &  \\ 
g^{\prime }(0)=g^{\prime }(R)=0 &  & 
\end{array}
\right. \,.  \label{g}
\end{equation}
Equation in \eqref{g} coincides with equation \eqref{eqf} by chosing $k=0$ and $\mu=\tau$.
Therefore the solutions to equation in  \eqref{g} are given by
$$
g(r)=f_0(r)= c_{1}r^{-\frac\beta2}
J_{\nu_0}(\sqrt{\tau_1 }r)
$$
with 
$\nu_0=\frac{\beta}{2}
=
\frac{N-2+\alpha}{2}$ 
and moreover, as in Section 5,  Neumann condition $g'(R)=0$ is equivalent to
\begin{equation}\label{neug}
-\frac\beta 2 J_{\nu_0}(\sqrt{\tau_1 }R) + \sqrt{\tau_1 }R J'_{\nu_0}(\sqrt{\tau_1 }R) =0\,.
\end{equation}
Furthemore we recall that  $\upsilon _{1}:=\upsilon_{1}(R)$ is the first eigenvalue of
\begin{equation}
\left\{ 
\begin{array}{ccc}
w^{\prime \prime }+\dfrac{N-1+\alpha }{r}w^{\prime }+\upsilon w-\dfrac{N-1}{
r^{2}}w=0 & \text{in} & (0,R) \\ 
&  &  \\ 
w(0)=w^{\prime }(R)=0. &  & 
\end{array}
\right.   \label{w}
\end{equation}
Equation in \eqref{w} coincides with equation \eqref{eqf} by chosing $k=1$ and $\mu=\upsilon$.
Therefore the solutions to equation in \eqref{w} are given by
$$
w(r)=f_1(r)= c_{1}r^{-\frac\beta2}
J_{\nu_1}(\sqrt{\upsilon_1 }r)
$$
with $\nu_1=\sqrt{N-1+\frac{\beta ^2}{4}
}=\sqrt{N-1+\frac{(N-2+\alpha) ^2}{4}
}
$ and  moreover, as in Section 5,   Neumann condition $w'(R)=0$ is equivalent to
\begin{equation}\label{neuw}
-\frac\beta 2 J_{\nu_1}(\sqrt{\upsilon_1 }R) + \sqrt{\upsilon_1 }R J'_{\nu_1}(\sqrt{\upsilon_1 }R) =0\,.
\end{equation}
By \eqref{neug} and \eqref{neuw} we deduce that $\sqrt{\tau_1 }R$ and  $\sqrt{\upsilon_1 }R$ are the smallest positive solution to the equations 
\begin{equation}\label{neug1}
-\frac\beta 2 J_{\nu_0}(x) + x J'_{\nu_0}(x) =0
\end{equation}
and 
\begin{equation}\label{neuw1}
-\frac\beta 2 J_{\nu_1}(x) + x J'_{\nu_1}(x) =0\,,
\end{equation}
respectively. Therefore, arguing as in Section 5,  since $\beta>0$, the positive root $x_{\nu_0, 1}$ of equation \eqref{neug1} coincides with the zero $j_{\nu_{\Io+1},1}$ of  the Bessel function $J_{\nu_\Io+1}$ and the positive root of \eqref{neuw1} coincides with $x_{\nu_1,1}$.
By properties of $x_{\nu, k}$, described in Section 5, we know that
\begin{equation}\label{1}
x_{\nu_1, 1}<j_{\nu_1,1}\,.
\end{equation}
Moreover by properties of zero's Bessel functions \eqref{zeros}, since for  $N\ge 2$ and $\alpha>0$, $\nu_1=\sqrt{N-1+\frac {\beta^2}4}<\nu_0+1
=
\frac{\beta}2+1$,  it results
\begin{equation}\label{2}
j_{\nu_1,1}<j_{{\nu_0+1},1}\equiv x_{\nu_\Io, 1}\,.
\end{equation}
Combining \eqref{1} and \eqref{2}, we get 
$$
x_{\nu_1,1}\equiv \sqrt{\upsilon_1 }R
< j_{\nu_1,1}< j_{\nu_0+1,1}\equiv x_{\nu_0,1} \equiv \sqrt{\tau_1 }R
\,.
$$
This yields the conclusion.
\end{proof}
We can now prove Theorem \ref{TSW}.

\vspace{.2 cm}

\noindent {\sl Proof of Theorem \ref{TSW}.} 
Lemma \ref{For_SW} ensures that 
$\mu_{1,\alpha}(\Omega ^{\sharp })$
 is a $N$-fold degenerate eigenvalue and a corresponding
set of eigenfunctions is
\begin{equation*}
w_{1}(\left\vert x\right\vert)\frac{x_{i}}{\left\vert x\right\vert},
\text{ for }i=1,...,N,
\end{equation*}
where $w_{1}$ is the first eigenfunction of problem (\ref{w}). As it is easy
to verify,  we have
\begin{equation}
\mu_{1,\alpha}(\Omega ^{\sharp })
=
\frac{\dint_{0}^{r^{\sharp }}\left[ \left( \dfrac{
d}{dr}w_{1}\right) ^{2}
+
\dfrac{N-1}{r^{2}}w_{1}^{2}\right] r^{\alpha +N-1}dr}{
\dint_{0}^{r^{\sharp }}w_{1}^{2}r^{\alpha +N-1}dr}. 
 \label{2_lambda}
\end{equation}
By the  assumptions on the symmetry of the set $\Omega $, it holds that
\begin{equation*}
\int_{\Omega }
w_{1}(\left\vert x\right\vert )
\frac{x_{i}}{\left\vert
x\right\vert } \left\vert x \right\vert^{\alpha}
dx = 0\text{ }\,\,\,  \forall i\in \left\{ 1,...,N\right\} .
\end{equation*}
Therefore we can use 
\begin{equation*}
w_{1}(\left\vert x\right\vert)
\frac{x_{i}}{\left\vert x\right\vert}, \,\,\,  \forall i=1,...,N,
\end{equation*}
as test functions for $\mu_{1,\alpha}(\Omega)$, obtaining 
\begin{equation}
\mu_{1,\alpha}(\Omega)
\leq 
\frac{\dint_{\Omega }\left[ 
\left(  G^{\prime }(\left\vert x\right\vert) \right) ^{2}\frac{x_{i}^{2}}{\left\vert x\right\vert^{2}}
+
\frac{G^{2}(\left\vert x\right\vert)}{\left\vert x\right\vert^{2}}
\left( 1-\frac{x_{i}^{2}}{
\left\vert x\right\vert^{2}}\right) \right] \left\vert x\right\vert ^{\alpha }dx}{\dint_{\Omega }
\left[ G^{2}(\left\vert x\right\vert)\frac{x_{i}^{2}}{\left\vert x\right\vert^{2}}\right] \left\vert x\right\vert
^{\alpha }dx}\text{ \ for \ }i=1,...,N\text{.}  \label{i}
\end{equation}
where
\begin{equation}
G(r):=\left\{ 
\begin{array}{ccc}
w_{1}(r) & \text{if} & r\leq r^{\sharp } \\ 
&  &  \\ 
w_{1}(r^{\sharp }) & \text{if} & r>r^{\sharp }.
\end{array}
\right.   \label{G}
\end{equation}
Summing over the index $i$ inequalities (\ref{i}), we get
\begin{equation*}
\mu_{1,\alpha}(\Omega))
\leq
 \frac{\dint_{\Omega }\left[
 \left(  G^{\prime } (\left\vert x\right\vert) \right) ^{2}
+
\dfrac{N-1}{\left\vert x\right\vert^{2}}G^{2}(\left\vert x\right\vert)\right]
 \left\vert x\right\vert^{\alpha }dx}
{\dint_{\Omega }G^{2}(\left\vert x\right\vert)\left\vert x\right\vert ^{\alpha }dx}
\text{.}
\end{equation*}

Note that, since $w_{1}^{\prime }(r)>0$ in $\left( 0,R\right) ,$ we have
that $G^{2}(r)$ is a non-decreasing function for $r\geq 0$. 
Hardy-Littlewood inequality (\ref{HL}), with $u=G^{2}$ 
and $v\equiv 1,$ yields
\begin{equation}
\int_{\Omega }G^{2}(\left\vert x\right\vert )
\left\vert x\right\vert
^{\alpha }dx
\geq
 \int_{\Omega ^{\sharp }}\left[ G^{2}(\left\vert
x\right\vert )\right] _{\sharp }
\left\vert x\right\vert ^{\alpha
}dx=\int_{\Omega ^{\sharp }}G^{2}(\left\vert x\right\vert )
\left\vert x\right\vert ^{\alpha }dx,  \label{G>}
\end{equation}
where the equality in (\ref{G>}) holds true thanks to the monotonicity of
the function $G^{2}(r).$

Now let us set 
\begin{equation}
N(r):=\left( \dfrac{d}{dr}G(r)\right) ^{2}+\dfrac{N-1}{r^{2}}G^{2}(r).
\label{N}
\end{equation}
Now we claim that the function $N(r)$ is strictly decreasing in $(0, + \infty)$.
Indeed we have
\begin{equation*}
\frac{d}{dr}N(r)
=
2G^{\prime }G^{\prime \prime }
+
\frac{2(N-1)}{r^{2}}GG^{\prime }
-
\frac{2(N-1)}{r^{3}}G^{2}.
\end{equation*}
Since $G^{\prime }(r)=0$ for any $r>r^{\sharp },$ we have
\begin{equation*}
\frac{d}{dr}N(r)
=-\frac{2(N-1)}{r^{3}}w_{1}^{2}(r^{\sharp })<0
\text{ \ for any } r > r^{\sharp } .
\end{equation*}
While for any $r\in (0,r^{\sharp })$ it holds
\begin{equation*}
\frac{d}{dr}N(r)
=
\frac{d}{dr}\left[ \left( \dfrac{d}{dr}w_{1}\right) ^{2}
+
\dfrac{(N-1)w_{1}^{2}}{r^{2}}\right] =2w_{1}^{\prime }w_{1}^{\prime \prime }
+
\frac{2(N-1)w_{1}w_{1}^{\prime }}{r^{2}}
-
\frac{2(N-1)}{r^{3}}w_{1}^{2}.
\end{equation*}
Using the equation for $w_{1}$ we get
\begin{equation*}
\frac{d}{dr}N(r) =
2w_{1}^{\prime }\left( -\dfrac{N-1+\alpha }{r}w_{1}^{\prime
}-
\mu_{1,\alpha}(\Omega^{\sharp }) w_{1}+\dfrac{N-1}{r^{2}}w_{1}\right) +2\frac{N-1
}{r^{2}}w_{1}w_{1}^{\prime }-\frac{2}{r^{3}}(N-1)w_{1}^{2}
\end{equation*}
\begin{equation*}
=-2\mu_{1,\alpha}(\Omega^{\sharp })
w_{1}^{\prime }w_{1}-2\dfrac{\alpha }{r}\left(
w_{1}^{\prime }\right) ^{2}-2\frac{N-1}{r}
\left[ 
 \left( w_{1}^{\prime
}\right) ^{2}-\dfrac{2}{r}w_{1}w_{1}^{\prime }+\frac{1}{r^{2}}
w_{1}^{2}
\right]
\end{equation*}
\begin{equation*}
=-2 \mu_{1,\alpha}(\Omega^{\sharp })
w_{1}^{\prime }w_{1}-2\dfrac{\alpha }{r}\left(
w_{1}^{\prime }\right) ^{2}
-2\frac{(N-1)}{r}
\left(
 w_{1}^{\prime }-\frac{
w_{1}}{r}
\right) ^{2}<0
\end{equation*}
Therefore 
\begin{equation*}
\frac{d}{dr}N(r)<0\text{ \ for any }r\in \left( 0,r^{\sharp }\right),
\end{equation*}
since we are assuming that $\alpha \in (0,N)$ and we know, by Lemma 7.1, that $w_{1}^{\prime }w_{1}\geq 0$ in $\left(
0,r^{\sharp }\right) .$
\\
By repeating the same arguments used for (\ref{G>}),  using the monotonicity of 
the function $N$, just proved, we get
\begin{equation}
\int_{\Omega }\left[ 
\left( G^{\prime}(\left\vert x\right\vert)\right) ^{2}
+
\dfrac{N-1}{\left\vert x\right\vert^{2}}
G^{2}(\left\vert x\right\vert)
\right] \left\vert x\right\vert ^{\alpha }dx
\leq
 \int_{\Omega ^{\sharp
}}\left[ 
\left( G^{\prime}(\left\vert x\right\vert)\right) ^{2}
+
\dfrac{N-1}{\left\vert x\right\vert^{2}}
G^{2}(\left\vert x\right\vert)
\right]
\left\vert x\right\vert ^{\alpha }dx.  
\label{N<}
\end{equation}
Inequalities (\ref{N<}) and (\ref{G>}), 
taking into account equality (\ref{2_lambda}), yield (\ref{SW}).

\noindent Finally, the proof easily shows that if 
$\mu_{1,\alpha}(\Omega ^{\sharp } )=\mu_{1,\alpha}(\Omega )$
then $\Omega \equiv \Omega^\sharp$. 
$\hfill \Box $

\vspace{.2 cm}

\begin{remark} 
Some numerics would suggest that if one drops
the assumption on the sign of $\alpha $, then the function $N(r)$, in general, is no
longer decreasing.
\end{remark}

We now briefly discuss the relation between our Theorem and Weinberger's one. 
First of all, we stress that our estimate relies heavily on the weighted embedding theorems presented in Section 2, and in particular on Theorem 2.1. Without these, the minimum in \eqref{mu_1} might not be attained.

Let $B_{R}$ be the ball in $\mathbb{R}^{N}$ centered at the origin with
radius $R$. The starting point of Weinberger's approach consists in
observing that $\mu _{1}(B_{R})$ is an $N$-degenerate eigenvalue, and a
basis for the corresponding eigenspace has the form:
\begin{equation*}
f(\left\vert x\right\vert )\frac{x_{i}}{\left\vert x\right\vert },\text{
with }i\in \left\{ 1,...,N\right\} .
\end{equation*}
The function $f$, which, as is well known, can be expressed in terms of
Bessel functions, satisfies the condition $f^{\prime }(r)=0=f(0)$.
In the case addressed in the present paper, 
we had to prove that the same phenomenon occurs for
$\mu _{1}(B_{R};\left\vert x\right\vert ^{\alpha})$.
Weinberger then proceeds by using the following $N$ smooth test functions for 
$\mu _{1}(\Omega )$:
\begin{equation*}
P_{i}(x)=F(\left\vert x\right\vert )\frac{x_{i}}{\left\vert x\right\vert },
\text{ with }i\in \left\{ 1,...,N\right\} ,
\end{equation*}
where
\begin{equation*}
F(t)=\left\{ 
\begin{array}{c}
f(t)\text{ for }t\in \left[ 0,r\right] \text{ \ \ \ } \\ 
\\ 
f(r)\text{ for }t\in \left( r,+\infty \right) .
\end{array}
\right.
\end{equation*}
Note that he is allowed to do so, since it is always possible to choose the
origin such that the following $N$ orthogonality conditions are simultaneously fulfilled:
\begin{equation}
 \label{N_oth}
P_{i}(x)\perp 1\text{ for any }i\in \left\{ 1,...,N\right\} .
\end{equation}
Since $\mu_{1, \alpha}(\Omega)$ changes when the origin is shifted, we add the hypothesis of
the symmetry of  $\Omega$ to ensure that the orthogonality conditions \eqref{N_oth} remain satisfied.
In fact, this is the only point where such an assumption is needed. 

Weinberger finally uses the previous considerations to express $\mu_1(\Omega)$ as the ratio of integrals of radial functions. These functions exhibit the appropriate monotonicity to ultimately obtain the estimate. 
Finally, we had to prove that similar circumstances also arise in our case.
 
 \bigskip

\noindent {\bf Acknowledgements:}
The first author wants to thank the University of Naples Federico II for a visiting appointment and the kind hospitality. The second author thanks the University of Leipzig for kind hospitality.  Part of this work was done when the third author was hosted at ``Institut de Mathématiques de Jussieu-Paris Rive Gauche, projet Combinatoire et Optimisation"; the author thanks this institution for the warm hospitality. The fourth author thanks Universit\' e de Le Havre and University of Rostock for their warm hospitality. 

\noindent F.Chiacchio and A.Mercaldo are members of GNAMPA of INdAM.

\noindent The authors thank the anonymous reviewers for their helpful and constructive comments and Nikita Simonov for suggesting some useful references.

\bigskip

\noindent {\bf Funding information:}
The research of F. Brock was supported by the University of Rostock. The research of F. Chiacchio  was partially supported  by the projects: PRIN 2017JPCAPN (Italy) Grant: Qualitative and quantitative aspects of nonlinear PDEs; PRIN PNRR 2022 - P2022YFAJH - Linear and Nonlinear PDE's: New directions and Applications. This work is partially supported by the ANR projects SHAPO and  STOIQUES financed by the French Agence Nationale de la Recherche (ANR).  The research of A.Mercaldo was partially supported by  Italian MIUR through research projects PRIN 2017 Direct and inverse problems for partial differential equations: theoretical aspects and applications,  PRIN 2022: PRIN20229M52AS Partial differential equations and related geometric-functional inequalities, PRIN PNRR 2022 - P2022YFAJH - Linear and Nonlinear PDE's: New directions and Applications. 

 \bigskip

\noindent {\bf Author contributions: } All authors have  contributed  equally to the  manuscript.
All authors have accepted responsibility for the entire content of this manuscript and consented to its submission to the journal, reviewed all the results, and approved the final version of the manuscript.

 \bigskip

\noindent {\bf Conflict of interest: } The authors state no conflict of interest.

\end{document}